\documentclass[11pt,a4paper]{amsart}
\usepackage{enumitem,xcolor,graphicx,hyperref}\usepackage[cp1250]{inputenc}
\usepackage{ifthen,amssymb,amsmath,subfig}\usepackage{amsthm}\usepackage{amsfonts}

\newtheorem{thm}{Theorem}[section]
\newtheorem{cor}[thm]{Corollary}
\newtheorem{lem}[thm]{Lemma}
\newtheorem{prop}[thm]{Proposition}
\theoremstyle{definition}
\newtheorem{defn}[thm]{Definition}

\numberwithin{equation}{section}
\DeclareMathOperator{\Arg}{\rm Arg}
\allowdisplaybreaks
\newcommand{\abs}[1]{\left\vert#1\right\vert}%
\newcommand{\myb}[2][0cm]{\mathopen{}\left[#2\parbox[h][#1]{0cm}{}\right]}
\newcommand{\myp}[2][0cm]{\mathopen{}\left(#2\parbox[h][#1]{0cm}{}\right)}
\newcommand{\set}[1]{\left\{#1\right\}}%
\begin{document}

\title[Analytic representation of the generalized Pascal snail and its applications]
{Analytic representation of the generalized Pascal snail and its applications}

\author[S. Kanas and V. S. Masih ]{S. Kanas$^{1}$ and V. S. Masih$^{2}$}

\address{$^{1}$University of Rzeszow, Al. Rejtana 16c, PL-35-959 Rzesz\'{o}w, Poland}
\email{skanas@ur.edu.pl}
%%%%%
\address{$^{2}$Department of Mathematics, Payame Noor University, Tehran, Iran}
\email{masihvali@gmail.com; v\_soltani@pnu.ac.ir}
\thanks{$^{1}$Corresponding author}

\begin{abstract}
We find an unifying approach to the analytic representation of the domain bounded by a generalized Pascal snail. 
Special cases as Pascal snail, Both leminiscate,  conchoid of the Sluze and a disc are included. 
The behavior of  functions related to generalized Pascal snail are demonstrated.
\end{abstract}
\maketitle

%===========================================================================================
\section{The analytic representation of a Pascal snail}\label{sec:3}
%===========================================================================================
For $-1 \le \alpha\le 1$, $-1\le \beta \le 1$, $\alpha\beta\neq \pm 1$, and $0\le \gamma <1$ let $\mathfrak{L}_{\alpha,\beta,\gamma}$ denote the complex valued mapping
\begin{equation}\label{funS}
\mathfrak{L}_{\alpha,\beta,\gamma}(z)=\frac{(2-2\gamma)z}{\myp{1-\alpha z}\myp{1-\beta z}}=\sum_{n=1}^{\infty}B_n
z^n=\left\{\begin{array}{lc}
(2-2\gamma)\displaystyle\sum_{n=1}^{\infty}\myp{\dfrac{\alpha^n-\beta^n}{\alpha-\beta}}z^n,\quad &\alpha\ne \beta;\\[1em]
(2-2\gamma)\displaystyle\sum_{n=1}^{\infty}n\alpha^{n-1}z^n,\quad& \alpha=\beta,
\end{array}\right.
\end{equation}
where $z\in \mathbb{D}=\{z\in\mathbb{C}:\  |z|< 1\}$.  We note that $\mathfrak{L}_{\alpha,\beta,\gamma}$ maps  $\mathbb{D}$ onto a domain $\mathfrak{D}(\alpha,\beta,\gamma)$ whose boundary is a  given by
\begin{multline*}
\partial\mathfrak{D}(\alpha,\beta,\gamma)=\Biggl\{w=u+i v\colon
 \frac{\myp{2(1-\gamma)u+(\alpha+\beta)(u^2+v^2)}^2}{(1+\alpha\beta)^2}+\frac{4(1-\gamma)^2v^2}{(1-\alpha\beta)^2}-(u^2+v^2)^2=0\Biggr\}.
 \end{multline*}
Indeed, for $z=e^{i \theta}$, with  $\theta\in[0,2\pi)$,  we obtain
\begin{align}
\mathfrak{L}(z):=\frac{z}{\myp{1-\alpha z}\myp{1-\beta z}}&=\frac{e^{i\theta}\myp{1-\alpha e^{-i \theta}}\myp{1-\beta e^{-i\theta}}}{\left|1-\alpha e^{i \theta}\right|^2\left|1-\beta e^{i \theta}\right|^2}\notag\\[0.5em]
&=\frac{\myp{1+\alpha\beta}\cos\theta-(\alpha+\beta)+i\myp{1-\alpha\beta}\sin\theta}{(1+\alpha^2-2\cos\theta)\myp{1+\beta^2-2\beta\cos\theta}},\ -1\le\alpha,\beta\le 1.\label{eq_funT_bou}
\end{align}
\color{black}Let
$u=u(\theta)=\Re\set{\mathfrak{L}\myp{e^{i \theta}}}$ and  $v=v(\theta)=\Im\set{\mathfrak{L}\myp{e^{i\theta}}}$.
Then
\begin{equation}\label{ReImS}
u=\frac{(1+\alpha\beta)\cos\theta-(\alpha+\beta)}{(1+\alpha^2-2\alpha\cos\theta)\myp{1+\beta^2-2\beta\cos\theta}}, \quad v=\frac{\myp{1-\alpha\beta}\sin\theta}{(1+\alpha^2-2\alpha\cos\theta)\myp{1+\beta^2-2\beta\cos\theta}}.\end{equation}
Hence,  $u,v$ satisfy the  equation
\begin{equation}\label{Pascal1}
 \frac{\myp{u+(\alpha+\beta)(u^2+v^2)}^2}{(1+\alpha\beta)^2}+\frac{v^2}{(1-\alpha\beta)^2}-(u^2+v^2)^2=0 \quad (\alpha\beta\neq\pm 1).
 \end{equation}
Therefore $\mathfrak{L}_{\alpha,\beta,\gamma}$ maps the unit circle onto a curve (cf. \cite{KT1, KT2})
\begin{equation}\label{Pascal2}
 \frac{\myp{2(1-\gamma)u+(\alpha+\beta)(u^2+v^2)}^2}{(1+\alpha\beta)^2}+\frac{4(1-\gamma)^2}{(1-\alpha\beta)^2}v^2=(u^2+v^2)^2\quad (\alpha\beta\neq\pm 1),
\end{equation}
or
\begin{equation}\label{Pascal3}
\left(u^2+v^2-\dfrac{2(1-\gamma)(\alpha+\beta)}{(1-\alpha^2)(1-\beta^2)}u\right)^2=\dfrac{4(1-\gamma)^2(1+\alpha\beta)^2}{(1-\alpha^2)^2(1-\beta^2)^2}
u^2+\frac{4(1-\gamma)^2(1+\alpha\beta)^2}{(1-\alpha^2)(1-\beta^2)(1-\alpha\beta)^2}v^2,
\end{equation}
that is generalization of \emph{the Pascal snail}  (see Fig.~\ref{Fig1} and Fig~\ref{Snail12}).

The  wide applications of the Pascal snail have been known since their description; the newest ones rely on the application to figure the path of airflow around object like plane wings, in the design of race and train tracks but also in cryptography for selecting the points of the curve (ellipse, leminiscate, etc.)  over the prime fields. Also, the leminiscate are used in the construction of grids on irregular regions  in the development of software for numerically solving partial differential equations.  Very recently a method based on leminiscate  is applied for meander like regions and rely on  covering the region with  sectors bounded by two confocal leminiscate and two arcs  orthogonal to the Pascal snail (cf. \cite{LCP}).

In this paper we  will deal with the Pascal snail \eqref{Pascal2} or \eqref{Pascal3} and its analytical representation. Also, we will discuss the special cases of \eqref{Pascal2} or \eqref{Pascal3} which give some interesting curves.\medskip
\begin{figure}[!ht]
\centering
\subfloat[$\alpha=-0.4,\beta= 0.9, \gamma=0.93$]{%
\includegraphics[width=0.4\textwidth]{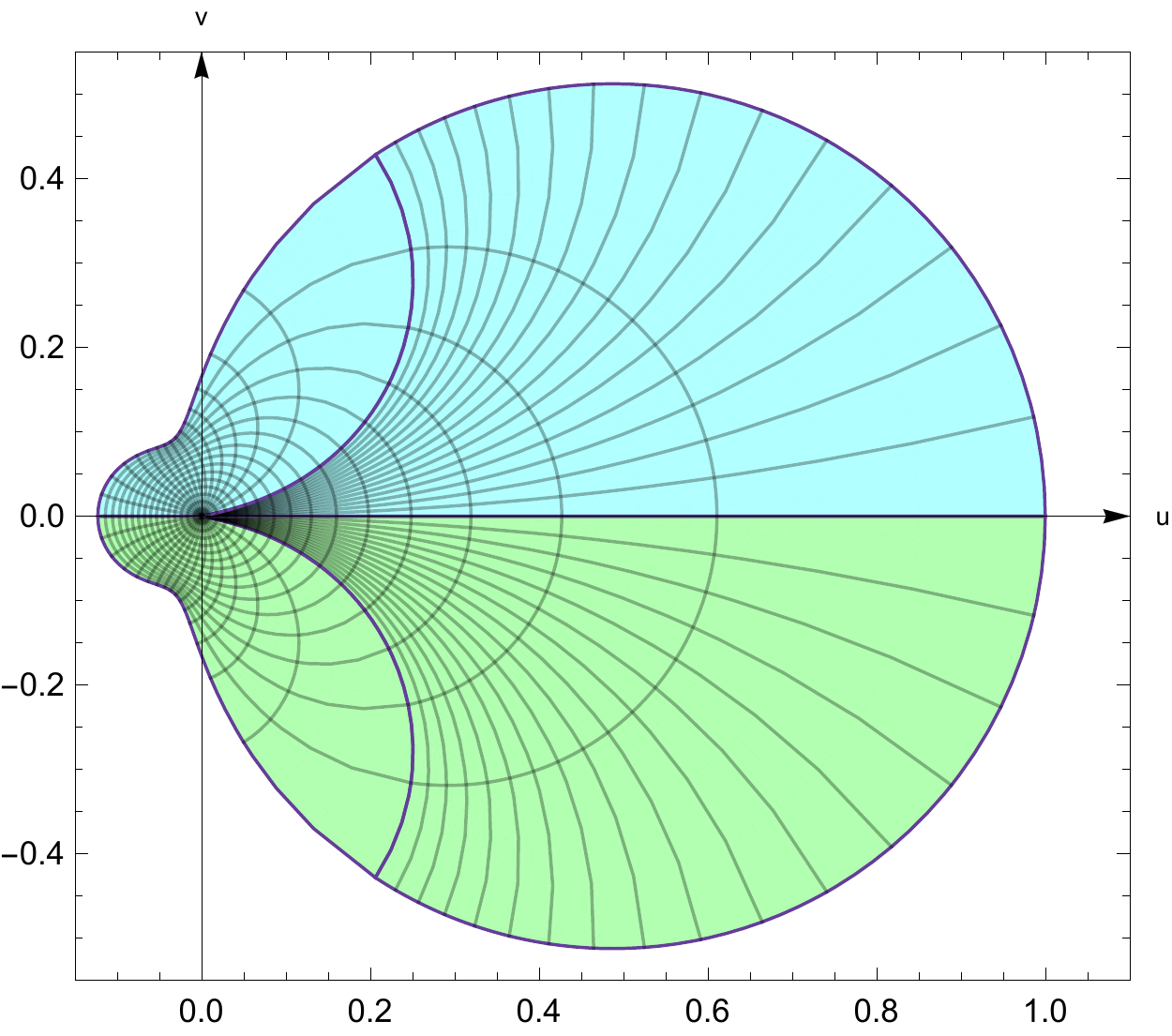}}%
\qquad\quad
\subfloat[$\alpha=-0.9, \beta=0.4, \gamma=0.93$]{%
\includegraphics[width=0.432\textwidth]{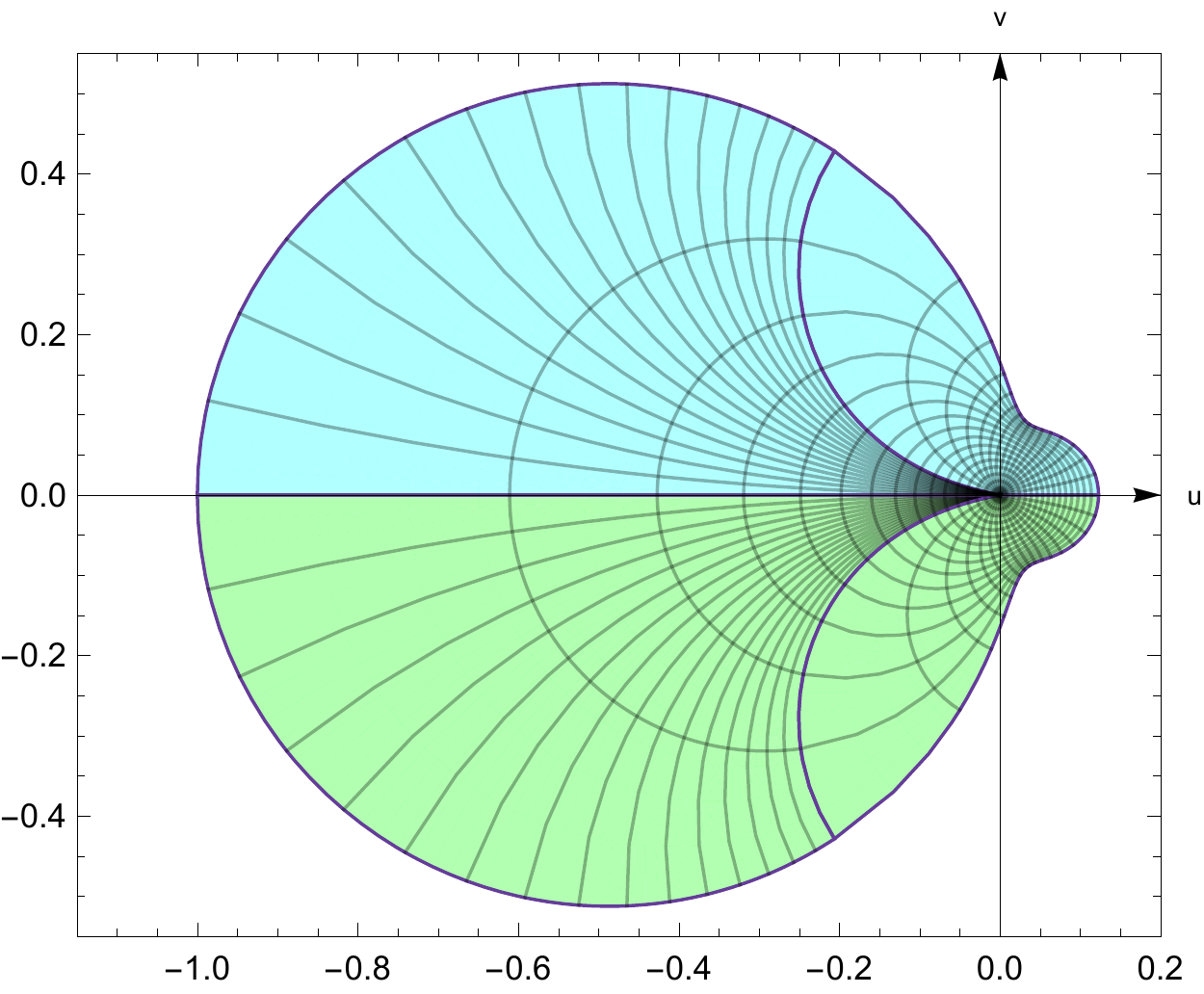}}%
\caption{The image of  $\mathbb{D}$ under $\mathfrak{L}_{\alpha,\beta,\gamma}(z)$}\label{Fig1}
\end{figure}
\begin{figure}[!ht]
\centering
\subfloat[$\alpha=-0.4,\beta=-0.9, \gamma=0.93$]{%
\includegraphics[width=0.4\textwidth]{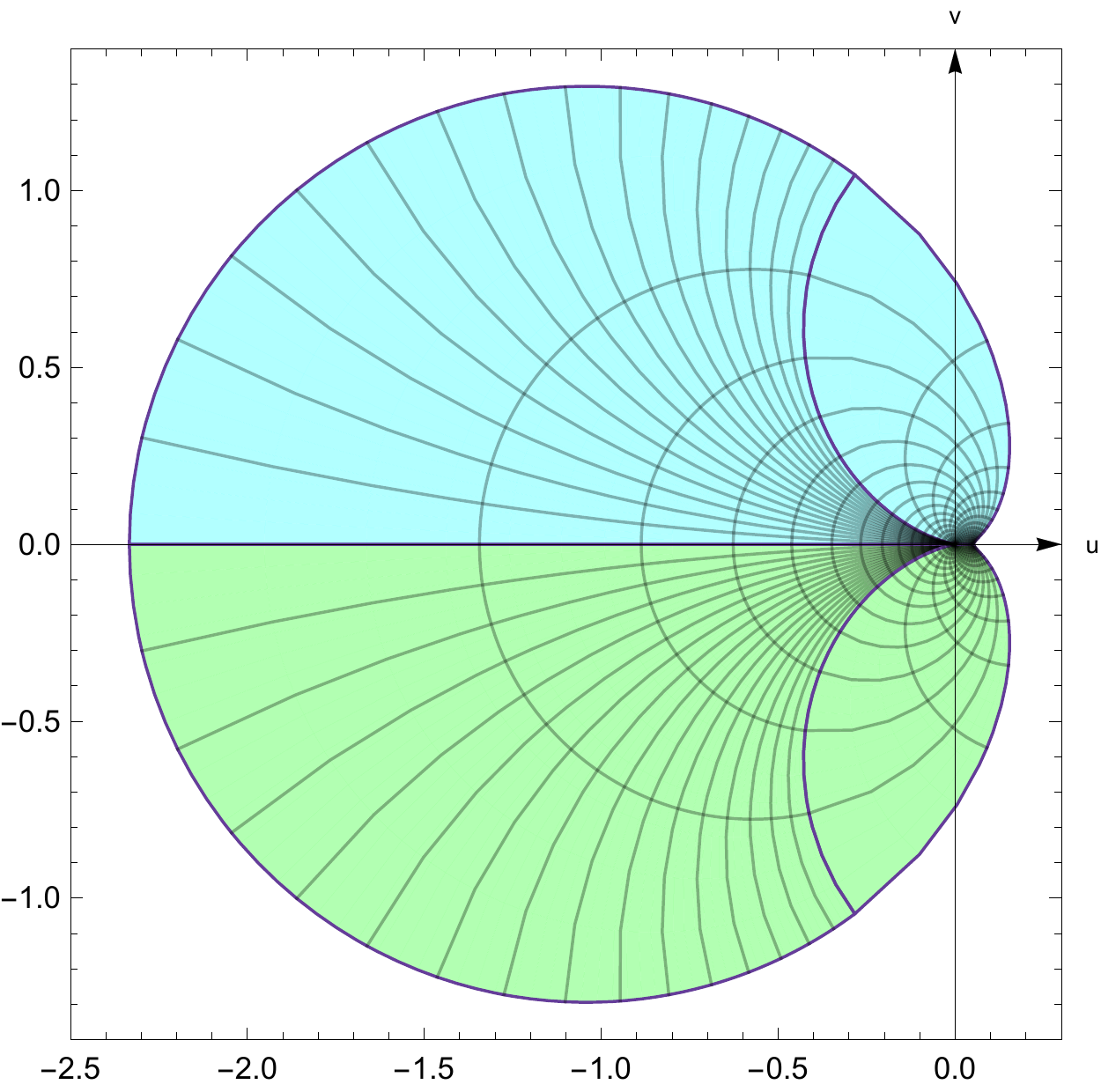}}%
\qquad\quad
\subfloat[$\alpha=0.9, \beta=0.4, \gamma=0.93$]{%
\includegraphics[width=0.4\textwidth]{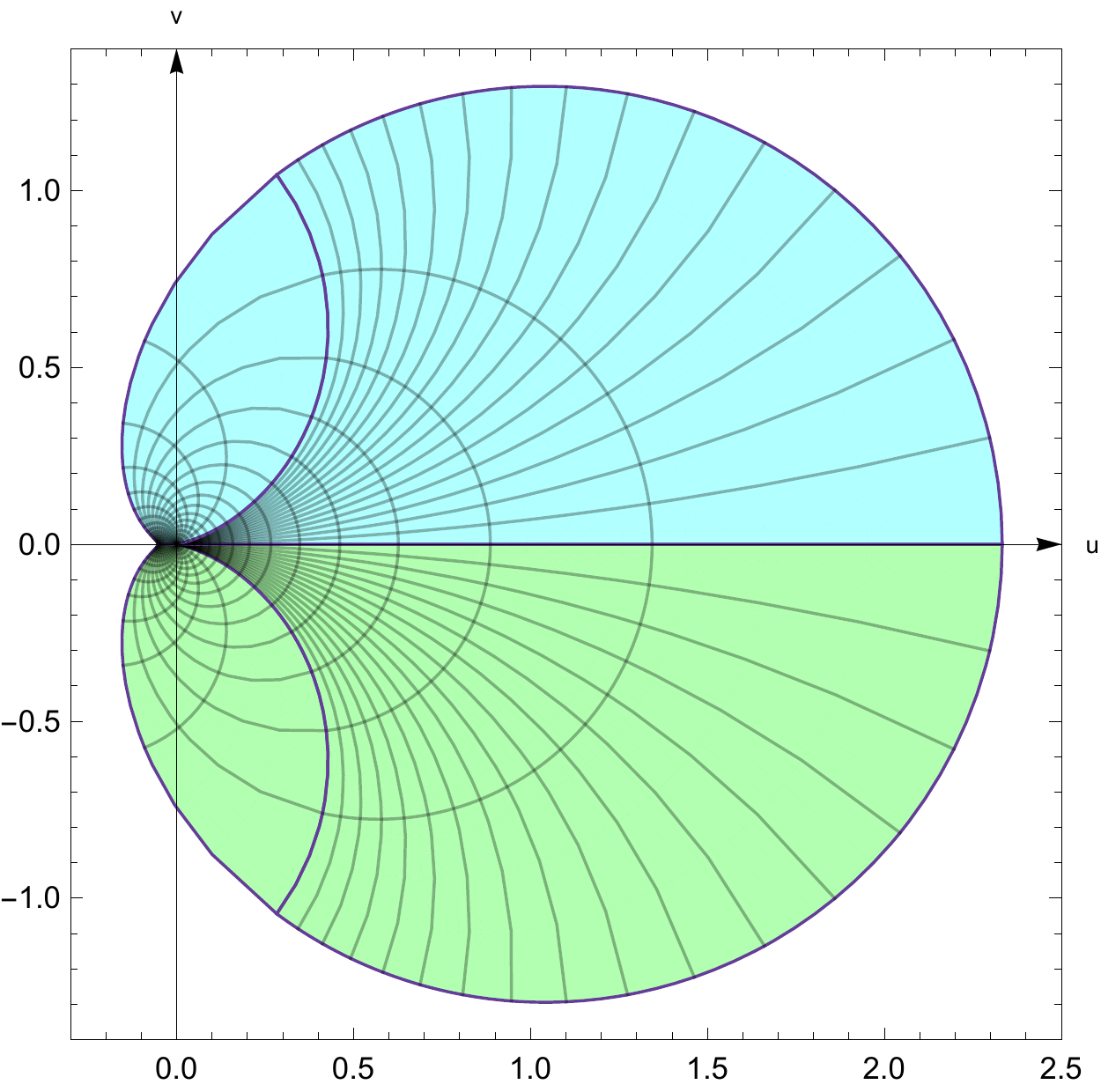}}%
\caption{The image of  $\mathbb{D}$ under $\mathfrak{L}_{\alpha,\beta,\gamma}(z)$}\label{Snail12}
\end{figure}
Let us  consider individual cases separately. By a symmetry, from now on we make the assumption: $\beta \ge \alpha$, unless otherwise stated.\medskip

\subsection{Circular domains} We get a circle for the case, when one of the parameter $\alpha$ or $\beta$ is zero, and the second is in the interval $(-1,1)$. Let  $\alpha=0<\beta < 1$. Then  $\mathfrak{L}_{0,\beta,\gamma}$ has the form $$\mathfrak{L}_{0,\beta,\gamma}(z)=\frac{2(1-\gamma)z}{1-\beta\, z},$$ and  $\mathfrak{L}_{0,\beta,\gamma}(\mathbb{D})$ is a circular domain
\begin{equation*}
\mathfrak{D}(0,\beta,\gamma)=\left\{w\in \mathbb{C}:\ \left|\frac{w}{2(1-\gamma)}-\frac{\beta}{1-\beta^2}\right|< \frac{1}{1-\beta^2}\right\}.
\end{equation*}

For the case $\alpha=\beta=0$ a curve $\mathfrak{D}(0,0,\gamma)$ is a circle $|w|<2(1-\gamma)$.

\subsection{Halfplane}  For the case when $\alpha = 0$ and $\beta =1$ the domain $\mathfrak{L}(\mathbb{D})$ is the halfplane $\Re w > \gamma-1$.
The case $\beta = 0, \alpha =-1$ gives a halfplane $\Re w < 1-\gamma$ which is not the ones of interest to us.

\subsection{Pascal snail regions}\label{Pascal snail regions} In the case $\beta =\alpha \in (-1,1)\setminus\{0\}$ the function $\ \mathfrak{L}_{\alpha,\alpha,\gamma}$ becomes
\begin{equation}\label{Lim1}
\mathfrak{L}_{\alpha,\alpha,\gamma}(z)=\frac{2(1-\gamma)z}{(1-\alpha z)^2},
\end{equation}
with $0\le \gamma <1$, that maps the unit disk onto simply connected and bounded region, which can be described as
\begin{equation}\label{Lim01}(u^2+v^2-eau)^2=a^2(u^2+v^2),\end{equation}
where
$$e=\frac{2\alpha}{1+\alpha^2},\quad a=\frac{2(1-\gamma)(1+\alpha^2)}{(1-\alpha^2)^2}.$$
The equation \eqref{Lim01} can be rewritten in a polar equation
\begin{equation}\label{Lim3}
\mathfrak{L}_{\alpha,\alpha,\gamma}(e^{it})=\set{\rho \mathrm{e}^{\mathrm{i}\varphi}\colon\
\rho = \Theta(\alpha,\varphi), \ -\pi<\varphi\leq\pi},
\end{equation}
where
$$\Theta(\alpha,\varphi)=\dfrac{2(1-\gamma)|\alpha|}{\alpha} \dfrac{1+\alpha^2+2|\alpha|\cos\varphi}{\myp{1-\alpha^2}^2}.$$

The boundary curve, known as Pascal snail (lima\c{c}on of Pascal), is a bicircular rational plane algebraic curve of degree 4 which belongs to the family of curves called centered trochoids or epitrochoids  (cf. Fig. \ref{Fig230.3}. Certainly $\mathfrak{L}_{0,0,\gamma}(\mathbb{D})$ is a disk). Pascal snail is the inversion of conic sections with respect to a focus.

\begin{figure}[!ht]
\centering
\subfloat[$\alpha=2-\sqrt{3},\gamma=0.5$]{%
\label{Fig230.31}\includegraphics[width=0.4\textwidth]{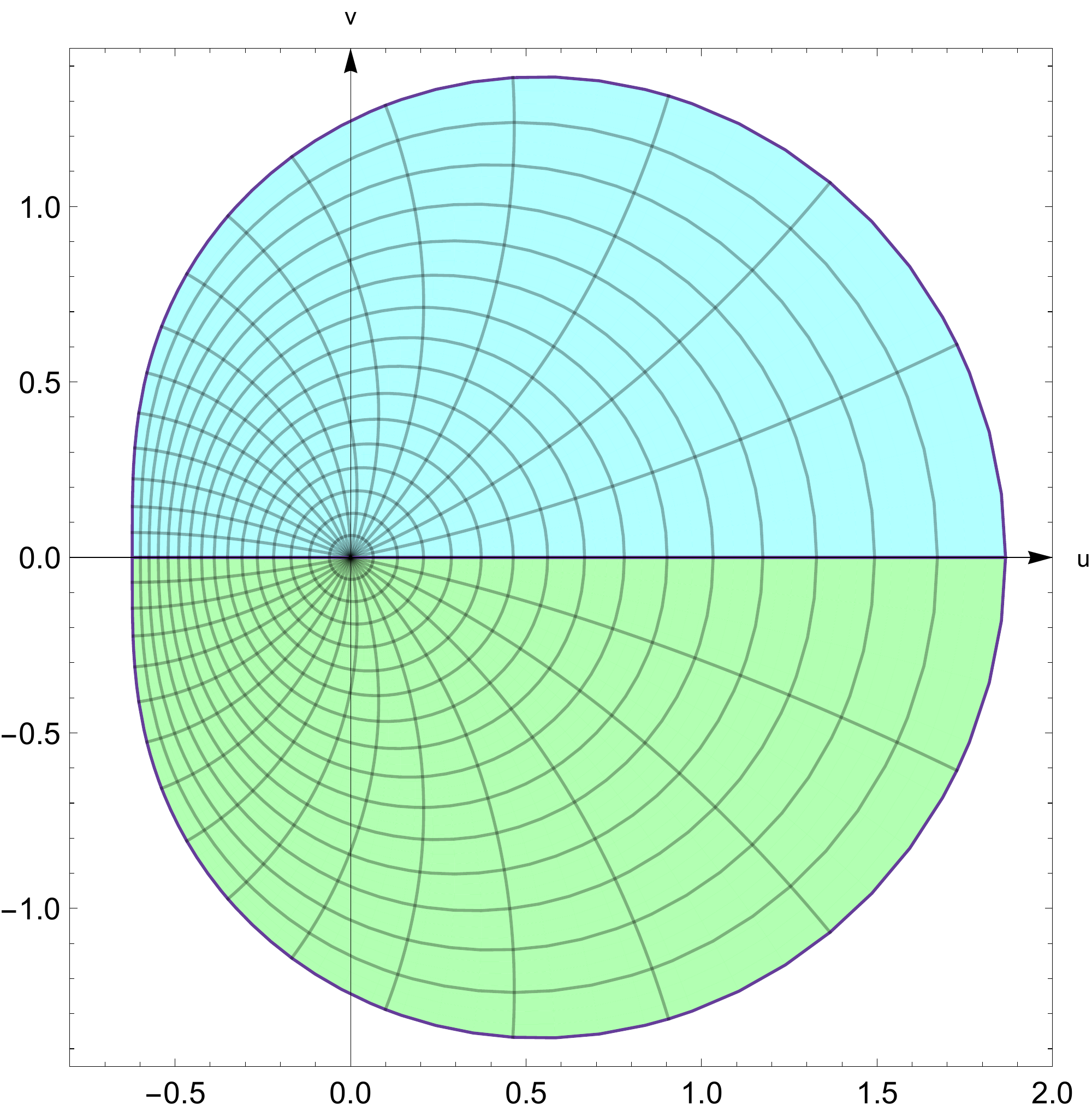}}%
\qquad\quad
\subfloat[$\alpha=-2+\sqrt{3},\gamma=0.5$]{%
\label{Fig230.32}\includegraphics[width=0.4\textwidth]{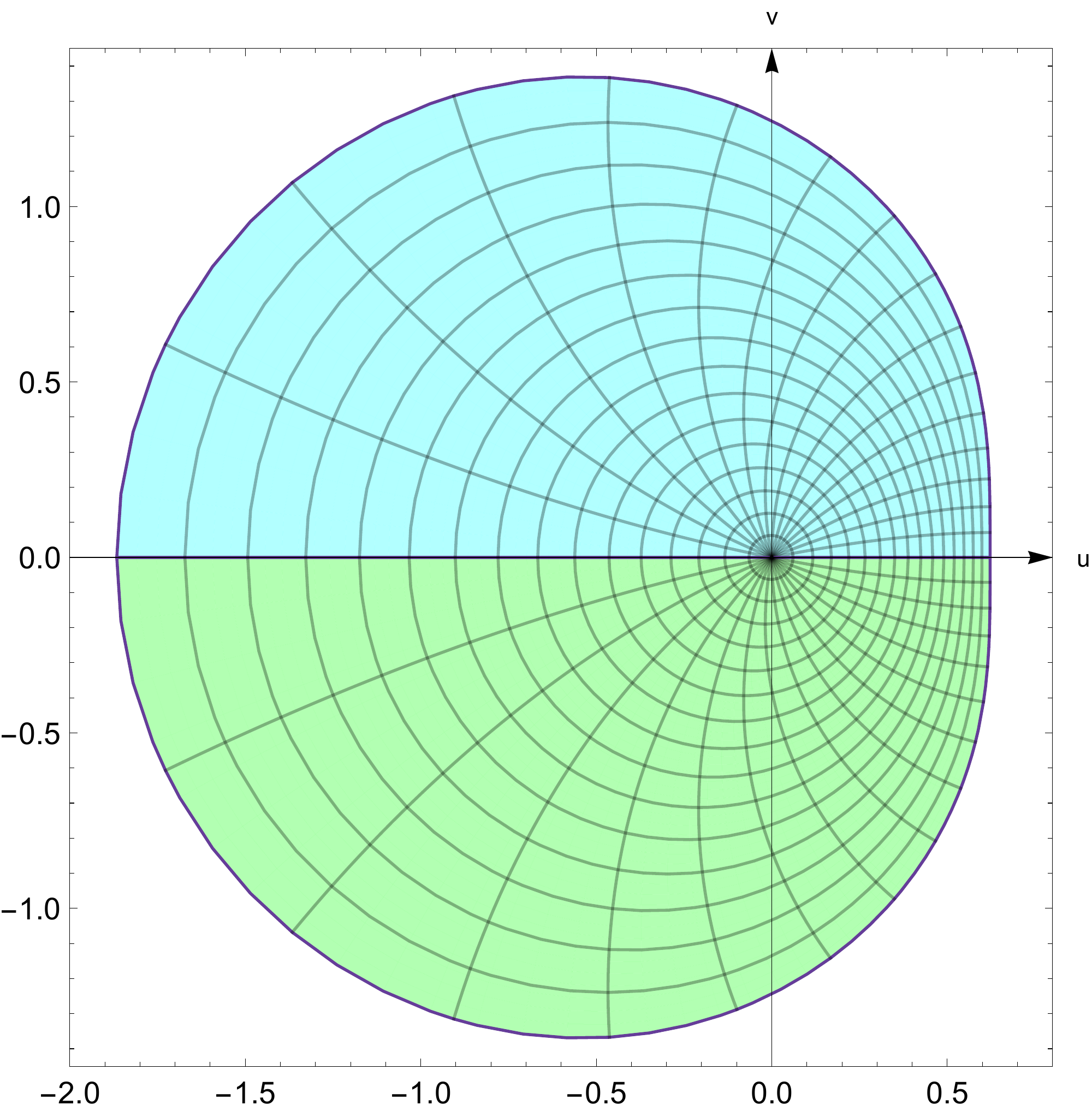}}%
\caption{The image of  $\mathbb{D}$ under $\mathfrak{L}_{\alpha,\alpha,\gamma}(z)$}\label{Fig230.3}
\end{figure}

We note that $|e|<1$ for $\alpha \neq 1$. In this case the snail is elliptic which is inverse of an ellipse with respect to its focus. In the case, when $|e|< 1/2$, that is $\alpha\in \left(-2+\sqrt{3},2-\sqrt{3}\right)$,  the domain  bounded by the Pascal snail \eqref{Lim01} is convex, and tends to the circle when $\alpha \to 0$.  For $|e|=1/2$ the snail has a flattened segment of the boundary and when $|e|>1/2$, that is for $\alpha \in \left(-1,-2+\sqrt{3}\right)\cup\left(2-\sqrt{3},1\right)$, the curve has a shape of a bean. The case when the Pascal snail has a loop does not hold, because it is equivalent to  the inequality $(1-\alpha)^2 <0$. Summarizing, the  domain  bounded by the Pascal snail is bounded, convex for $\alpha\in \left[-2+\sqrt{3},2-\sqrt{3}\right]$, concave for $\alpha \in \left(-1,-2+\sqrt{3}\right)\cup\left(2-\sqrt{3},1\right)$, and symmetric with respect to real axis.

 The function $\mathfrak{L}_{\alpha,\alpha,\gamma}(z)$ with $\alpha\ne 0$ can also be written as a composition of two analytic univalent functions, that is,
\[
\mathfrak{L}_{\alpha,\alpha,\gamma}(z)=\myp{h_2\circ h_1}(z)=\frac{1-\gamma}{2\alpha}\myb{\myp{\frac{1+\alpha z}{1-\alpha z}}^2-1},
\]
where
\[
h_1(z)=\frac{1+\alpha z}{1-\alpha z}\quad\textrm{and}\quad h_2(z)=\frac{1-\gamma}{2\alpha}\myp{z^2-1}.
\]
The function $h_1$ is univalent in $\mathbb{D}$ and $h_2$ is univalent in $h_1(\mathbb{D})=\left\{w\in \mathbb{C}\colon \left|\frac{w-1}{w+1}\right|<|\alpha|\right\}$.
\subsection{Conchoid of the Sluze} In the third special case we set $\beta =1, \alpha \in (-1,1)\setminus\{0\}$ or $\alpha =-1, \beta\in (-1,1)\setminus\{0\}$.  Let us consider $\beta =1, \alpha \in (-1,1)$. Thus $\mathfrak{L}_{\alpha,1,\gamma}$ has a form
\begin{equation}\label{Sluze1}
\mathfrak{L}_{\alpha,1,\gamma}(z)=\frac{2(1-\gamma)z}{(1-\alpha z)(1- z)},
\end{equation}
where $0\le \gamma <1$, that maps the unit disk onto simply connected region with boundary that is a curve
\begin{align*}\label{Sluze2}
\partial\mathfrak{D}(\alpha,1,\gamma)&=\left\{u+\mathrm{i} v\colon
\left(u+\frac{(1+\alpha)(1-\gamma)}{\myp{1-\alpha}^2}\right)(u^2+v^2)-\frac{4\alpha(1-\gamma)}{\myp{1-\alpha}^2\myp{1+\alpha}} u^2=0
\right\}\\
&=\Biggl\{ u+i v\colon \frac{\left[2(1-\gamma)u+(1+\alpha)(u^2+v^2)\right]^2}{(1+\alpha)^2}+\frac{4(1-\gamma)^2v^2}{(1-\alpha)^2}=\left[u^2+v^2\right]^2\Biggr\}
\end{align*}
known as the Conchoid of de Sluze, see Fig.~\ref{Fig12}. We note that the special case $\beta=1$ and $-1<\alpha<0$ was also considered  in \cite{MEY}.
\begin{figure}[!ht]
	\centering
	\subfloat[$\alpha=-0.5,\gamma=0.5$]{%
		\includegraphics[height=5.5cm,width=0.4\textwidth]{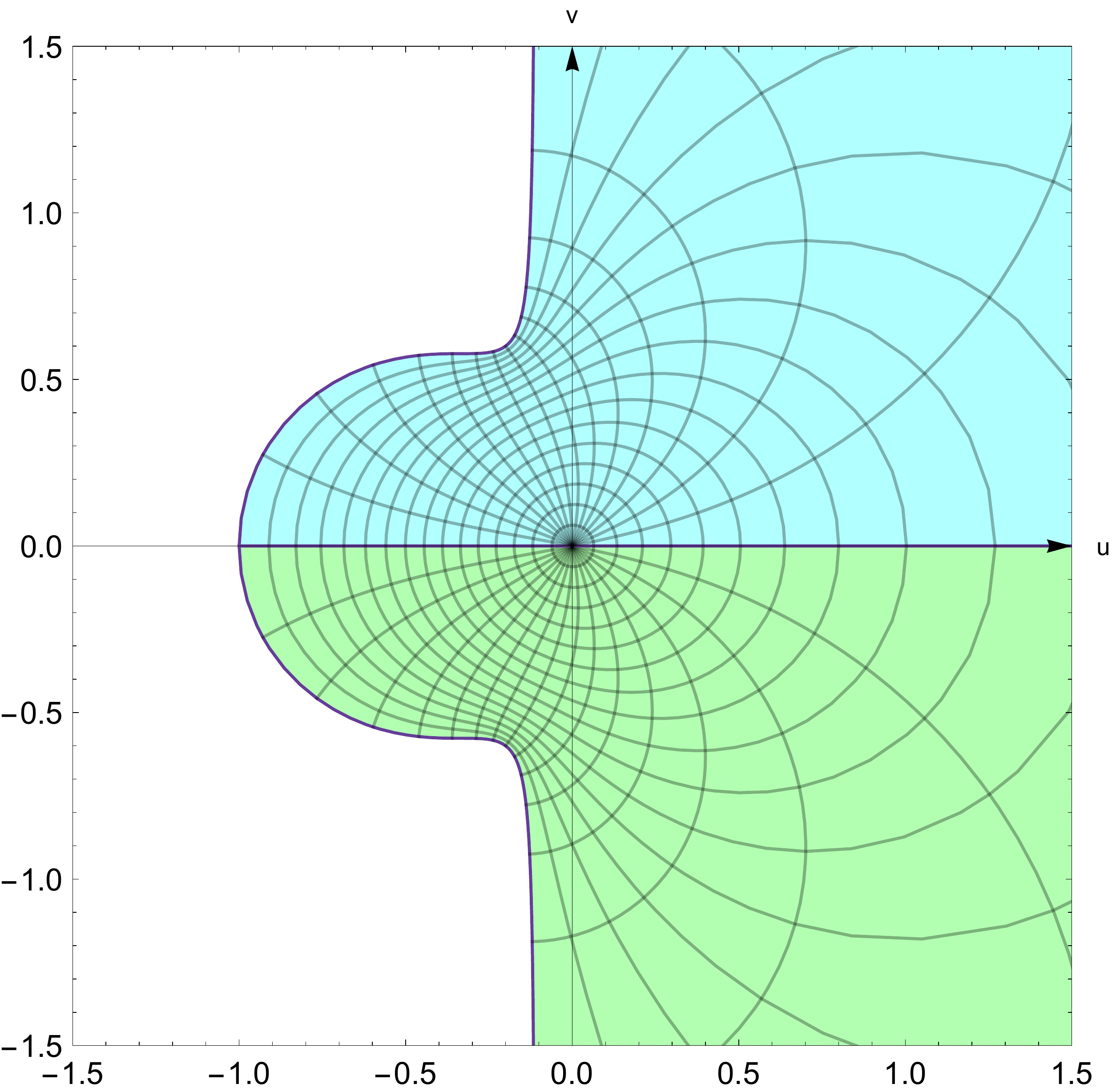}}%
	\qquad\quad
	\subfloat[$\alpha=0.5, \gamma=2/3$]{%
		\includegraphics[height=5.5cm,width=0.4\textwidth]{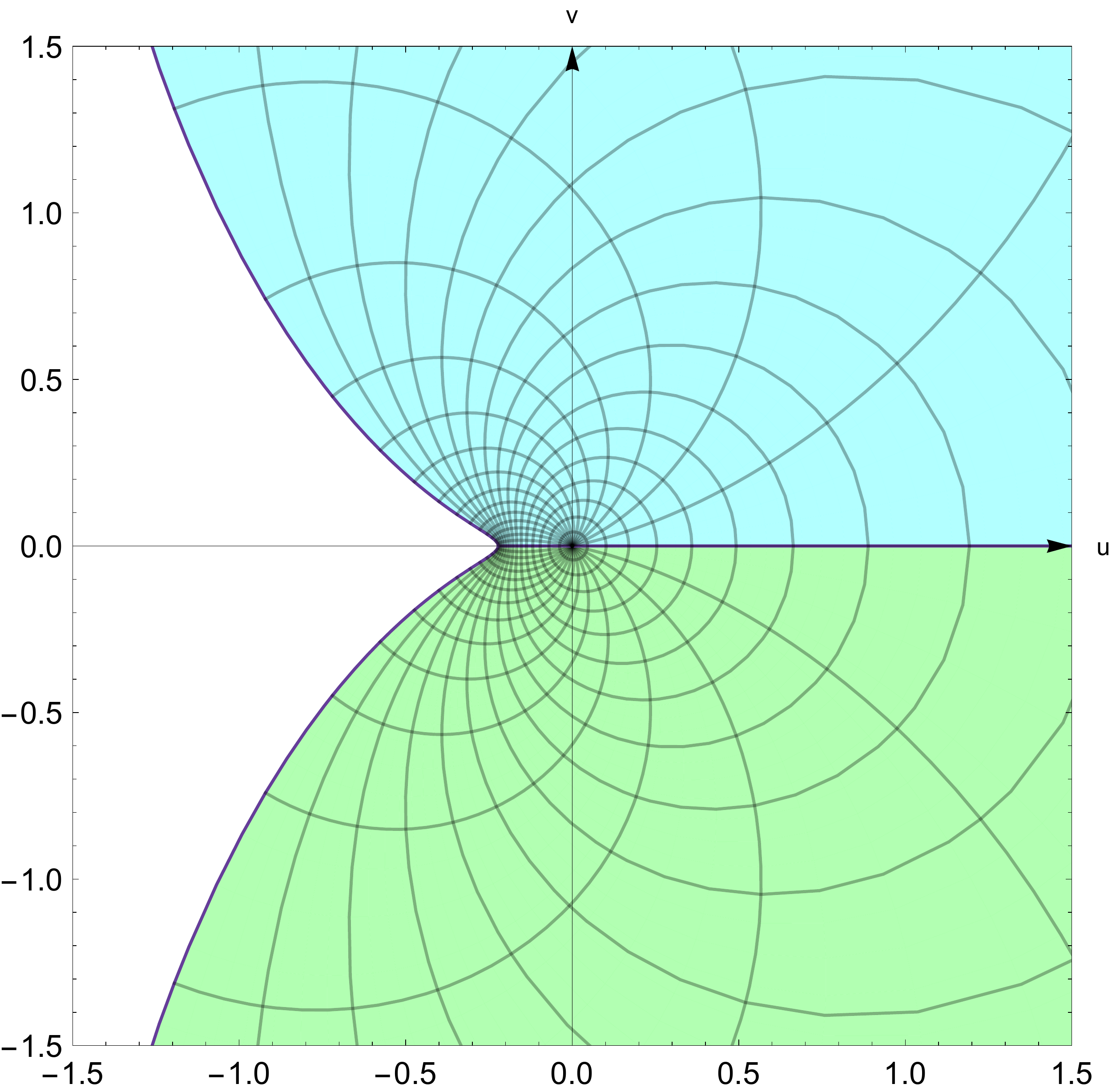}}%
	\caption{The image of  $\mathbb{D}$ under $\mathfrak{L}_{\alpha,1,\gamma}$}\label{Fig12}
\end{figure}

In the case when $\alpha =-1, \beta\in (-1,1)\setminus\{0\}$ we obtain the conchoid  of de Sluze (see Fig.~\ref{Fig34})  of the form
\begin{figure}[!ht]
	\centering
	\subfloat[$\beta=-0.5,\gamma=2/3$]{%
		\includegraphics[height=5.5cm,width=0.4\textwidth]{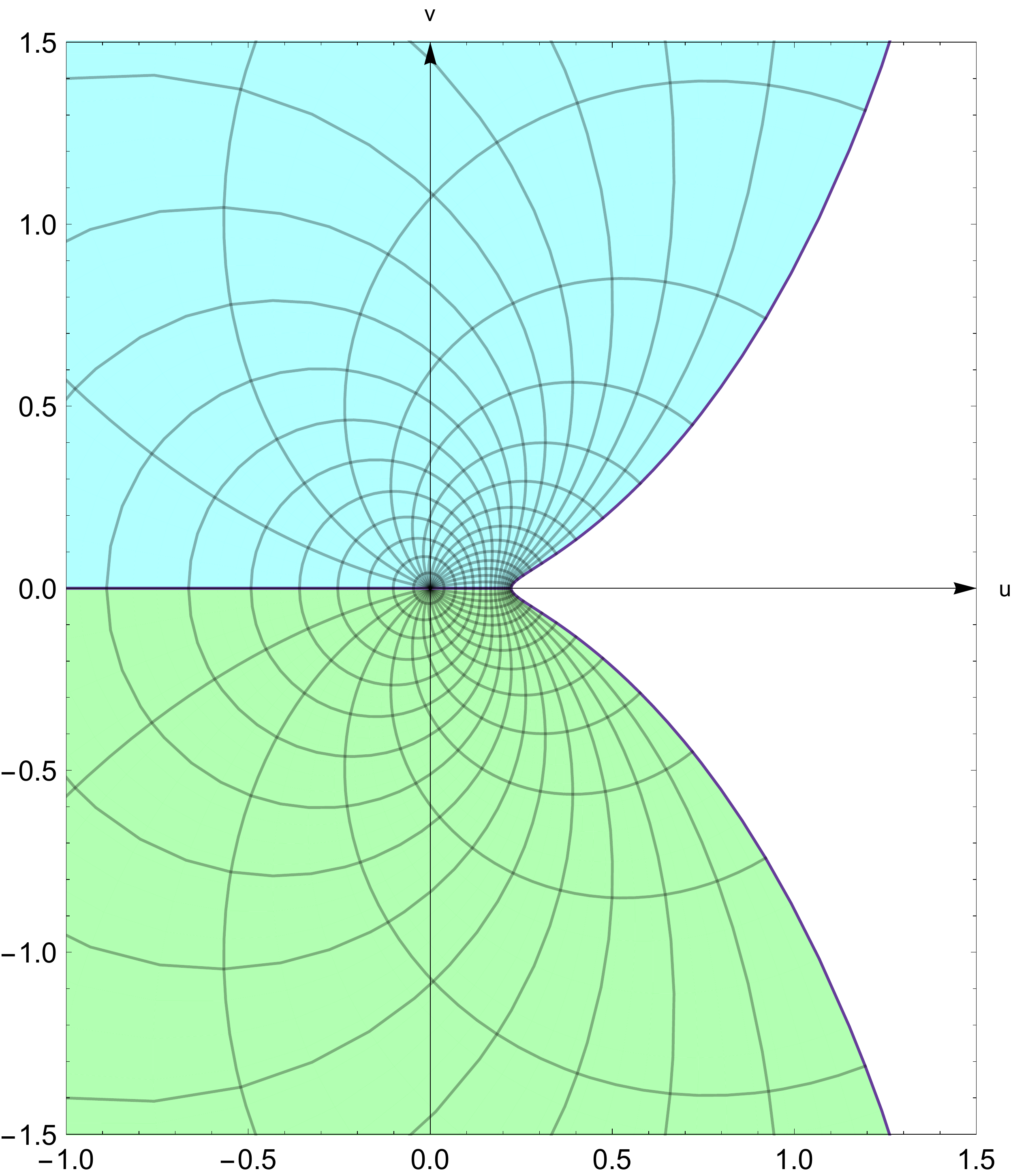}}%
	\qquad\quad
	\subfloat[$\beta=\gamma=0.5$]{%
		\includegraphics[height=5.5cm,width=0.4\textwidth]{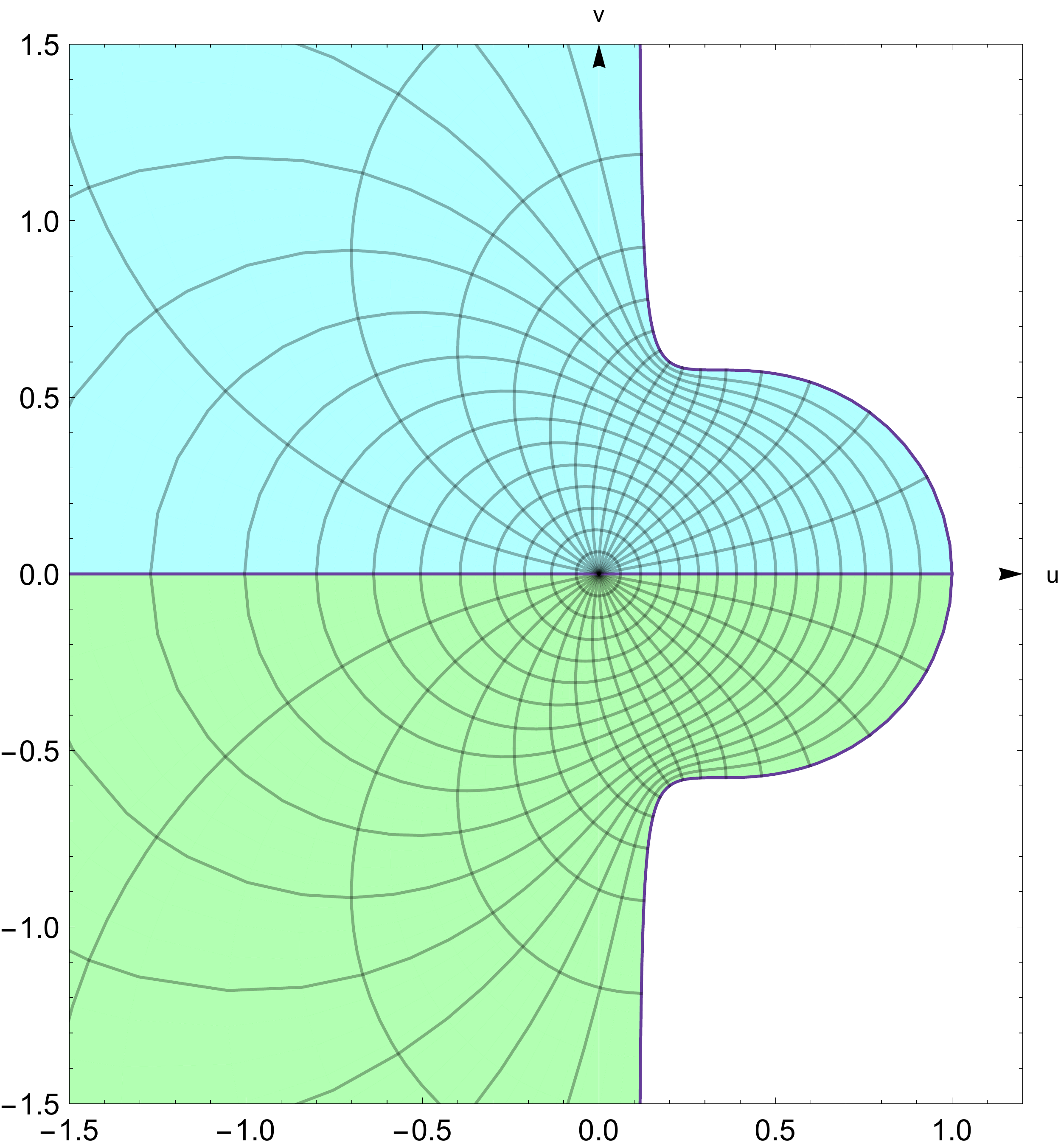}}%
	\caption{The image of  $\mathbb{D}$ under $\mathfrak{L}_{-1,\beta,\gamma}$}\label{Fig34}
\end{figure}

\begin{align*}
\partial\mathfrak{D}(-1,\beta,\gamma)&=\left\{u+\mathrm{i} v\colon
\myp{u-\frac{(1-\gamma)(1-\beta)}{\myp{1+\beta}^2}}(u^2+v^2)-\frac{4\beta(1-\gamma)}{\myp{1+\beta}^2\myp{1-\beta}}\,
u^2=0\right\}\\
&=\Biggl\{ u+i v\colon \frac{\left[2(1-\gamma)u-(1-\beta)(u^2+v^2)\right]^2}{(1-\beta)^2}+\frac{4(1-\gamma)^2v^2}{(1+\beta)^2}=\left[u^2+v^2\right]^2 \Biggr\},
\end{align*}
symmetric to the $\partial\mathfrak{D}(\alpha,1,\gamma)$ with respect to the imaginary axis.

\subsection{Hippopede. Leminiscate of Booth} Here we let  $\beta =-\alpha, \alpha \in (-1,0)$. In this case $\mathfrak{L}_{\alpha,-\alpha,\gamma}$ is of the form
\begin{equation}\label{Both1}
\mathfrak{L}_{\alpha,-\alpha,\gamma}(z)=\frac{2(1-\gamma)z}{1-\alpha^2 z^2},
\end{equation}
and the equation \eqref{Pascal2} or \eqref{Pascal3} reduces to  $(u^2+v^2)^2 = c^2u^2+d^2v^2$, with $c=2(1-\gamma)/(1-\alpha^2), d=2(1-\gamma)/(1+\alpha^2)$, that is
\begin{equation}\label{Both_1}
\partial\mathfrak{D}(\alpha,-\alpha,\gamma)=\Biggl\{u+i v\colon
\frac{4(1-\gamma)^2}{(1-\alpha^2)^2}u^2+\frac{4(1-\gamma)^2}{(1+\alpha^2)^2}v^2=\left(u^2+v^2\right)^2\Biggr\}.
\end{equation}
We remind that the hippopede is the bicircular rational algebraic curve of degree $4$,  symmetric with respect to both axes. Any hippopede is the intersection of a torus with one of its tangent planes that is parallel to its axis of rotational symmetry. When $c> d>0$ (that is $\alpha \neq 0$)  such a curve is known as an oval or leminiscate of Booth,  see Fig. \ref{Booth12}. Since the case $d=-c$ does not hold, the leminiscate \eqref{Both_1} do not reduce to the leminiscate of Bernoulli.
\begin{figure}[!ht]
	\centering
	\subfloat[$\alpha=-\sqrt{2/3}, \gamma=0.5$]{%
		\label{Fig230.32}\includegraphics[width=0.56\textwidth]{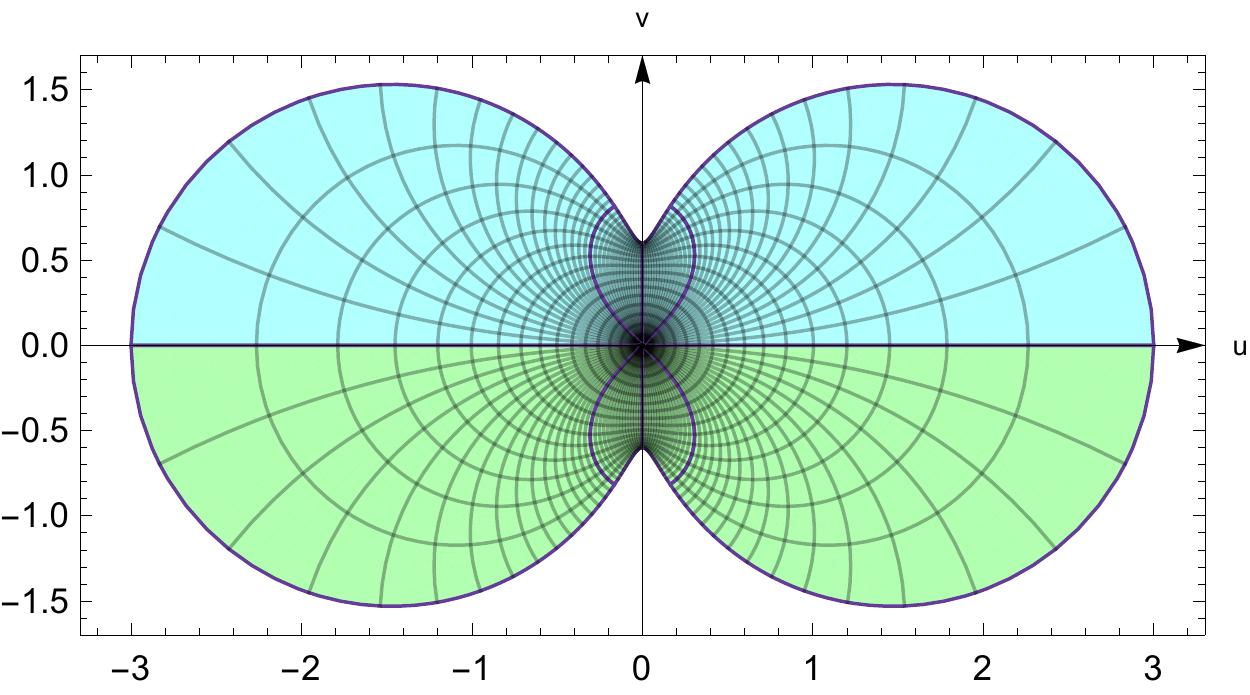}}%
	\caption{The image of  $\mathbb{D}$ under $\mathfrak{L}_{\alpha,-\alpha,\gamma}(z)$}\label{Booth12}
\end{figure}

We note that for $c/\sqrt{2}<d<c\sqrt{2}$,  the domain  bounded by the hippopede  is convex,  that is for $-\sqrt{3-2\sqrt{2}}< \alpha<0$, and the curve is called Booth's oval.  For $d=c/\sqrt{2}$, that is for $\alpha=- \sqrt{3-2\sqrt{2}}$ the hippopede has a flattened segment of the boundary. Summarizing, the  domain  bounded by the hippopede  is bounded and convex for $\alpha\in \left[-\sqrt{3-2\sqrt{2}},0\right)$, and concave for $\alpha \in \left(-1, -\sqrt{3-2\sqrt{2}}\right)$.

\subsection{Remaining cases} In this case, we consider the remaining range of parameters, i.e. $-1< \alpha < \beta < 1$, which were not considered in previous subsections. In these cases  a curve $\mathfrak{L}_{\alpha,\beta,\gamma}(e^{it})$ is the generalized Pascal snail, that has the form
\begin{equation}\label{Pascal5}(u^2+v^2-au)^2 = c^2u^2+d^2v^2\end{equation} with
\begin{equation}\label{acd}
a=\frac{2(1-\gamma)(\alpha+\beta)}{(1-\alpha^2)(1-\beta^2)},\ c=\dfrac{2(1-\gamma)(1+\alpha\beta)}{(1-\alpha^2)(1-\beta^2)},\ d=\frac{2(1-\gamma)(1+\alpha\beta)}{(1-\alpha\beta)\sqrt{(1-\alpha^2)(1-\beta^2)}}.
\end{equation}
We note, that the curve represented by the equation \eqref{Pascal5} has similar properties to the Pascal snail, and is symmetric only with respect to real axis. It has either horizontal eight-like shape,  bean-shape, pear-shape  or is convex. From this reason the region bounded by \eqref{Pascal5} is convex,  or concave. As we can see in the Theorem \ref{Th_Re} the minimum and maximum of real part are not always achieved on the real axis. Taking into account the geometrical properties of set $\mathfrak{L}(\mathbb{D})$, we get the following.

\begin{thm}[\cite{KT2}]\label{Th_Re} Let $-1 \le\alpha\le \beta\le 1$ and $\alpha\beta\neq \pm 1$. Then
$$\max\limits_{0\le \theta< 2\pi}\Re\, \mathfrak{L}(e^{i\theta}) =
\left\{\begin{array}{lcl}
\frac{(1+\alpha\beta)^2}{2(1-\alpha\beta)[2\sqrt{\alpha\beta(1-\alpha^2)(1-\beta^2)}-(\alpha+\beta)(1-\alpha\beta)]}&\textit{for}&(\alpha,
\beta)\in B_2,\\
\frac{1}{(1-\alpha)(1-\beta)}&& \textit{otherwise}, \\
\end{array}\right.$$

$$\min\limits_{0\le \theta< 2\pi}\Re\, \mathfrak{L}(e^{i\theta}) =
\left\{\begin{array}{lcl}
\frac{-(1+\alpha\beta)^2}{2(1-\alpha\beta)[2\sqrt{\alpha\beta(1-\alpha^2)(1-\beta^2)}+(\alpha+\beta)(1-\alpha\beta)]}&\textit{for}& (\alpha,\beta)\in B_1, \\
\frac{-1}{(1+\alpha)(1+\beta)}& & \textit{otherwise}, \\
\end{array}\right.$$
where
$$B_1=\left\{0 <\alpha<1,\  \beta_1(\alpha) < \beta <1
\right\},\quad B_2 =\left\{-1 <\alpha<0, \ \alpha < \beta<\beta_2(\alpha)\right\},$$ with
$$\beta_1(\alpha)=\frac{(1+\alpha) \sqrt{\alpha^2 + 14 \alpha + 1 }-(\alpha^2 + 6 \alpha + 1)}{2\alpha(1-
\alpha)},$$
$$\beta_2(\alpha) =\frac{(1-\alpha)\sqrt{ \alpha^2 - 14 \alpha+1 }-\alpha^2 + 6 \alpha -1}{2\alpha(1+
\alpha)}.$$
The sets $B_1, B_2$ are represented on a Fig. \ref{Fig2}.
\begin{figure}[h]
	\centering
	\includegraphics[width=0.4\textwidth]{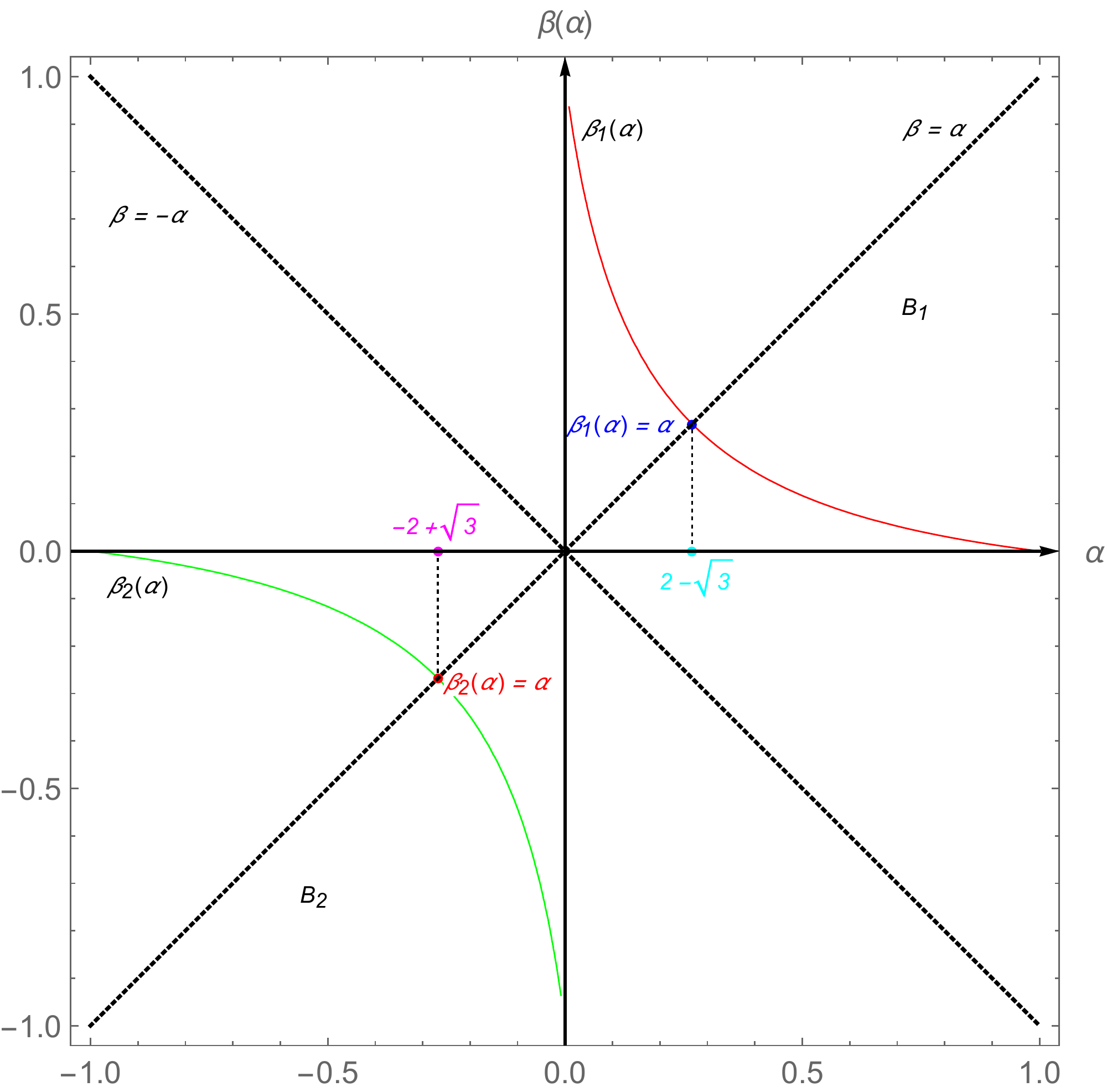}%
	\caption{The range of the parameters $\alpha, \beta$.}\label{Fig2}
\end{figure}

In the  special cases we have
$$\max\limits_{0\le \theta< 2\pi}\Re\, \mathfrak{L}(e^{i\theta}) =
\left\{\begin{array}{lcl}
-\frac{1}{8\alpha}\left(\frac{1+\alpha^2}{1-\alpha^2}\right)^2&\textit{for}& -1<\alpha\le \sqrt{3}-2,\ \beta =\alpha, \\
\frac{1}{(1-\alpha)^2}&\textit{for}&\sqrt{3}-2\le\alpha <1,\ \beta =\alpha, \\
- \frac{1+\alpha}{2(1-\alpha)^2}  & \textit{for}& -1< \alpha\le0,\ \beta=1,\\
-  \frac{1}{2(1+\alpha)} &\textit{for}& 0\le \alpha<1,\ \beta=1,\\
\frac{1}{1-\alpha^2}&\textit{for}&-1<\alpha <0,\ \beta=-\alpha.
\end{array}\right.$$

$$\min\limits_{0\le \theta< 2\pi}\Re\, \mathfrak{L}(e^{i\theta}) =
\left\{\begin{array}{lcl}
-\frac{1}{8\alpha}\left(\frac{1+\alpha^2}{1-\alpha^2}\right)^2&\textit{for}& 2-\sqrt{3}\le \alpha<1,\ \beta=\alpha,\\
-\frac{1}{(1+\alpha)^2}&\textit{for}&-1<\alpha \le\sqrt{3}-2,\ \beta =\alpha, \\
- \frac{1}{2(1+\alpha)}  &\textit{for}& -1< \alpha\le0,\ \beta=1,\\
- \frac{1-\alpha}{2(1+\alpha)^2}&\textit{for}& 0\le \alpha<1,\ \beta=1,\\
-\frac{1}{1-\alpha^2}&\textit{for}&-1<\alpha <0,\ \beta=-\alpha.
\end{array}\right.$$
\end{thm}

\subsection{Conclusions}
In general $\mathfrak{L}(\mathbb{D})$ is a domain symmetric about the real axis and starlike with respect to origin and such that $\mathfrak{L}(0)=0$, $\mathfrak{L}'(0)=1 > 0$. The geometrical properties of the regions $\mathfrak{L}(\mathbb{D})$ provides a natural bridge between the convex and concave domains. We also note that such domains were discussed in relation of generalized typically-real functions and generalized Chebyshev polynomials of the second kind \cite{KT1, KT2}.

From Theorem \ref{Th_Re} we conclude the following Corollary.
%-----------------------------------------------------------------------------------------------------------------
 \begin{cor}\label{Cor1_Th-Re}
Let $-1 \le\alpha\le \beta\le 1$, $\alpha\beta\neq \pm 1$ and $0\le\gamma<1$, and let $\mathfrak{L}_{\alpha,\beta,\gamma}$ be the function defined by \eqref{funS}. Then, for   $z\in \mathbb{D}$ we have

$$\Re\set{\mathfrak{L}_{\alpha,\beta,\gamma}(z)}> \mathfrak{L}_0(\alpha,\beta,\gamma)=
\left\{\begin{array}{lcl}
\frac{-(1+\alpha\beta)^2(1-\gamma)}{(1-\alpha\beta)[2\sqrt{\alpha\beta(1-\alpha^2)(1-\beta^2)}+(\alpha+\beta)(1-\alpha\beta)]}&\textit{for}& (\alpha,\beta)\in B_1, \\
\frac{-2(1-\gamma)}{(1+\alpha)(1+\beta)}& & otherwise, \\
\end{array}\right.$$
and
$$ \Re\set{\mathfrak{L}_{\alpha,\beta,\gamma}(z)}<\mathfrak{M}_0(\alpha,\beta,\gamma) =\left\{\begin{array}{lcl}
\frac{(1+\alpha\beta)^2(1-\gamma)}{(1-\alpha\beta)[2\sqrt{\alpha\beta(1-\alpha^2)(1-\beta^2)}-(\alpha+\beta)(1-\alpha\beta)]}&\textit{for}&(\alpha,
\beta)\in B_2,\\
\frac{2(1-\gamma)}{(1-\alpha)(1-\beta)}&& otherwise, \\
\end{array}\right.$$
Also, if $\alpha=0$, then
$$\left|\mathfrak{L}_{0,\beta,\gamma}(z)-\frac{2(1-\gamma)\beta}{1-\beta^2} \right|\le \frac{2(1-\gamma)}{1-\beta^2}.
$$
\end{cor}
%=========================================================================================================
In the sequel we will use the following lemma.

 \begin{lem}[\cite{Bab}]\label{inequalityreal}
Let $z$ is a complex number with positive real part. Then for any real
number $t$ such that $t\in [0,1]$, we have
$\Re\set{z^t}\ge \myp{\Re\, z}^t$.
\end{lem}
%===========================================================================================
\section{Subclass of the Carath\`{e}odory class related to the generalized Pascal snail}\label{sec:2}
%===========================================================================================
 Denote by $\mathcal{P}$ the Carath\`{e}odory class of functions i.e. $\mathcal{P}= \{p:\ p(z)=1+p_1 z+ p_2 z^2+\cdots,\ \Re\, p(z) > 0\ (z\in \mathbb{D})\}$. The fundamental importance of  $\mathcal{P}$ in geometric functions theory relies on the construction of several related families of analytic functions and is well known. Hence, various subclasses of $\mathcal{P}$ were defined and studied. Classical cases are related to the  halfplane and angular domain i.e.  $\mathcal{P}(\alpha)$ that denotes a subclass of  $\mathcal{P}$ consisting with functions with real part greater than $\alpha\ (0\le \alpha < 1)$, and  $\mathcal{P}_\gamma$ the class with argument between $-\gamma\pi/2$ and $\gamma\pi/2\ (0<\gamma\le 1)$.
Also, several subfamilies of  $\mathcal{P}$ were determined by the fact that some functionals are contained in convex subdomains of right halfplane. Therefore any subfamily of halfplane domains were considered in the context to a subfamily of $\mathcal{P}$. Hence a definition of the domains related to the Pascal snail was a motivation to the definition of some subclass of $\mathcal{P}$ associated with such domains.  To do this we first translate a domain $\mathfrak{L}_{\alpha,\beta,\gamma}(\mathbb{D})$  with a vector $(1,0)$ in order to obtain a domain $\mathfrak{D}_{\alpha,\beta,\gamma}$ contained in a right halfplane such that $1\in \mathfrak{D}_{\alpha,\beta,\gamma}$.
The boundary of the domain $\mathfrak{D}_{\alpha,\beta,\gamma}$ is described as follows:

$\partial\mathfrak{D}_{\alpha,\beta,\gamma}\quad=\Bigg\{u+iv:\ \dfrac{\myp{(2-2\gamma)(u-1)+(\alpha+\beta)((u-1)^2+v^2)}^2}{(1+\alpha\beta)^2}$

$\hspace{7cm} +\dfrac{4(1-\gamma)^2v^2}{(1-\alpha\beta)^2}-((u-1)^2+v^2)^2=0\Bigg\}.$

 We note that $\mathfrak{D}_{\alpha,\beta,\gamma}$ is contained in a halfplane $\Re\, w > 1+\mathfrak{L}_0$, where $\mathfrak{L}_0$ is given in Corollary \ref{Cor1_Th-Re}. Anyway, there is substantial difference between $\mathfrak{D}_{\alpha,\beta,\gamma}$ and a halfplane because  $\mathfrak{D}_{\alpha,\beta,\gamma}$ is not always a convex domain. However, when $\alpha= 0$ and $\beta\to 1^-$ then $\mathfrak{D}_{\alpha,\beta,\gamma}$ tends to a halfplane $\Re \, w  > \gamma-1$.
Thus $\mathfrak{D}_{\alpha,\beta,\gamma}$ provides a natural bridge between the convex and the concave domains.

Now, we define a function $\mathcal{T}_{\alpha,\beta,\gamma}$ as
 \begin{equation}\label{eq_def_funT}
\mathcal{T}_{\alpha,\beta,\gamma}(z)=1+ \mathfrak{L}_{\alpha,\beta,\gamma}(z),
\end{equation}
that map $\mathbb{D}$ univalently onto a domain $\mathfrak{D}_{\alpha,\beta,\gamma}$. Rewriting Corollary  \ref{Cor1_Th-Re}  for the function $\mathcal{T}_{\alpha,\beta,\gamma}$ we conclude the following theorem.
%-----------------------------------------
\begin{thm}\label{main1}
Let $-1<\alpha\le\beta<1$, $0\le \gamma<1$, and let $\mathcal{T}_{\alpha,\beta,\gamma}(\cdot)$ be defined by \eqref{eq_def_funT}. Then
$$\Re\set{\mathcal{T}_{\alpha,\beta,\gamma}(z)}> 1+\mathfrak{L}_0(\alpha,\beta,\gamma)$$
and
$$ \Re\set{\mathcal{T}_{\alpha,\beta,\gamma}(z)}<1+\mathfrak{M}_0(\alpha,\beta,\gamma),$$
where $\mathfrak{L}_0(\alpha,\beta,\gamma)$ and $\mathfrak{M}_0(\alpha,\beta,\gamma)$ are given in Corollary \ref{Cor1_Th-Re}.
\end{thm}
Now, we are ready to construct a class $\mathcal{P}_{snail}(\alpha,\beta,\gamma)$ as follows
$$\mathcal{P}_{snail}(\alpha,\beta,\gamma)=\{p\in \mathcal{P}:\ p(\mathbb{D})\subset \mathfrak{D}_{\alpha,\beta,\gamma}\}=\{p\in \mathcal{P}:\ p\prec \mathcal{T}_{\alpha,\beta,\gamma}\}.$$

%===========================================================================================
\section{The classes $\mathcal{ST}_{snail}(\alpha,\beta,\gamma)$,  $\ \mathcal{CV}_{snail}(\alpha,\beta,\gamma)$ and their properties}\label{sec:4}
%===========================================================================================
In this Section we give a concise presentation of some families of analytic functions related to the generalized Pascal snail $\mathcal{T}_{\alpha,\beta,\gamma}$.
We will study  some subclasses of $\mathcal{S}$ with functions analytic and univalent in $\mathbb{D}$ of the form
\begin{equation}\label{funf}
f(z)=z+\sum_{n=2}^{\infty}a_nz^n \quad (z \in \mathbb{D}).
\end{equation}
We also recall a class $\mathcal{ST}(\beta)\subset \mathcal{S}$, called \emph{starlike functions of order $0\le \beta<1$}, that consist of functions $f$ satisfying a condition
$$\Re\set{{zf'(z)}/{f(z)}}>\beta\quad (z \in \mathbb{D})$$
and a class $\mathcal{CV}(\beta)$, called \emph{convex functions of order $0\le \beta<1$}, with analytic condition
$$ \Re\set{1+{zf''(z)}/{f'(z)}}>\beta \quad (z \in \mathbb{D}).$$
\par
Let $f$ and $g$ be analytic in $\mathbb{D}$. Then the function $f$ is said to \emph{subordinate} to $g$ in $\mathbb{D}$ written by $f(z)\prec g(z)$,
if there exists a self-map of the unit disk $\omega$, analytic in $\mathbb{D}$ with $\omega(0)=0$ and  such that $f(z)=g(\omega(z))$. If $g$ is univalent in $\mathbb{D}$, then $f\prec g$ if and
only if $f(0)=g(0)$ and $f(\mathbb{D})\subset g(\mathbb{D})$.
\par
Also, let $\mathcal{ST}[\beta]$ be the subclass of $\mathcal{ST}$ defined by
\[
\mathcal{ST}[\beta]:=\left\{f\in \mathcal{A}\colon \ \frac{zf'(z)}{f(z)}\prec\frac{1}{1-\beta z} \right\}
\]
where $-1\le\beta\le 1, \beta\ne 0$. Notice that for $\beta=\pm 1$ the function $w=1/(1-\beta z)$ maps the unit disc $\mathbb{D}$ onto the half-plane
$\Re w>1/2$, and for $-1<\beta<1$  the function $w=1/(1-\beta z)$ maps the unit disc $\mathbb{D}$ onto the disc $D(C(\beta), R(\beta))$ with the center $C(\beta)=1/(1-\beta^2)$  and the radius $R(\beta)=|\beta|/(1-\beta^2)$.
%==========================================
\begin{lem}\label{univalent}
Let $-1<\alpha\le\beta<1$, $0\le\gamma<1$, and  $\mathfrak{L}_{\alpha,\beta,\gamma}$ be  defined by \eqref{funS}. Then $\mathfrak{L}_{\alpha,\beta,\gamma}$ is  starlike  in $\mathbb{D}$, moreover
\[\frac{\mathfrak{L}_{\alpha,\beta,\gamma}(z)}{2-2\gamma}\in \mathcal{ST}\left(\frac{1-|\alpha\beta|}{\myp{1+|\alpha|}(1+|\beta|)}\right)\quad \textit{and}\quad \frac{\mathfrak{L}_{\alpha,\beta,\gamma}(z)}{2-2\gamma}\in \mathcal{CV}\left(t_0(\alpha,\beta)\right),
\]
where
\begin{equation}\label{t_0}
0\le t_0(\alpha,\beta)=
\left\{
	\begin{array}{ll}
\dfrac{1-|\alpha|}{1+|\alpha|}+\dfrac{1-|\beta|}{1+|\beta|}-\dfrac{1+\alpha\beta}{1-\alpha\beta}\quad\textit{for}\quad \alpha\beta\ge0,\\
\dfrac{1-|\alpha|}{1+|\alpha|}+\dfrac{1-|\beta|}{1+|\beta|}-\dfrac{1-\alpha\beta}{1+\alpha\beta}\quad\textit{for}\quad \alpha\beta<0.
	\end{array}\right.
\end{equation}
Also, if $|z|=r<1$, then (see Fig.~\ref{Fig1}, Fig.~\ref{Snail12} and Fig.~\ref{Booth12})
$$\max_{|z|=r}\left|\mathfrak{L}_{\alpha,\beta,\gamma}(z)\right| =
\left\{\begin{array}{lcl}
\mathfrak{L}_{\alpha,\beta,\gamma}(r)&\textit{for}& \alpha\beta>0 \ \textit{with}\ \alpha+\beta>0\ \textit{or}\ \alpha=0, \\
-\mathfrak{L}_{\alpha,\beta,\gamma}(-r)&\textit{for}&\alpha\beta>0 \ \textit{with}\ \alpha+\beta<0 \ \textit{or}\ \beta=0, \\
 \mathfrak{L}_{\alpha,\beta,\gamma}(r) & \textit{for}& \alpha\beta<0 \ \textit{with}\ \alpha+\beta>0,\\
-\mathfrak{L}_{\alpha,\beta,\gamma}(-r) &\textit{for}& \alpha\beta<0 \ \textit{with}\ \alpha+\beta<0,\\
\left|\mathfrak{L}_{\alpha,-\alpha,\gamma}(\pm r)\right| &\textit{for}&  \alpha+\beta=0,
\end{array}\right.$$
$$\min_{|z|=r}\left|\mathfrak{L}_{\alpha,\beta,\gamma}(z)\right| =
\left\{\begin{array}{lcl}
-\mathfrak{L}_{\alpha,\beta,\gamma}(-r)&\textit{for}& \alpha\beta>0 \ \textit{with}\ \alpha+\beta>0 \ \textit{or}\  \alpha=0, \\
\mathfrak{L}_{\alpha,\beta,\gamma}(r)&\textit{for}&\alpha\beta>0 \ \textit{with}\ \alpha+\beta<0 \ \textit{or}\  \beta=0, \\
\frac{4(1-\gamma)r\sqrt{|\alpha\beta|}}{(\beta-\alpha)(1-\alpha\beta r^2)} & \textit{for}& \alpha\beta<0 \ \textit{with}\ \alpha+\beta>0,\\
\frac{4(1-\gamma)r\sqrt{|\alpha\beta|}}{(\beta-\alpha)(1-\alpha\beta r^2)} &\textit{for}& \alpha\beta<0 \ \textit{with}\ \alpha+\beta<0,\\
\left|\mathfrak{L}_{\alpha,-\alpha,\gamma}(\pm i r)\right| &\textit{for}&  \alpha+\beta=0.
\end{array}\right.$$
\end{lem}
\begin{proof}
A straightforward calculation shows that $G:=\mathfrak{L}_{\alpha,\beta,\gamma}$ satisfy
\begin{equation*} \Re\set{\frac{zG'(z)}{G(z)}}=1+ \Re\set{\frac{\alpha z}{1-\alpha  z}}+ \Re\set{\frac{\beta z}{1-\beta  z}}>1-\frac{|\alpha|}{1+|\alpha|}-\frac{|\beta|}{1+|\beta|},
\end{equation*}
from which the result concerning starlikeness follows.

In addition, we have
\[
1+\frac{zG''(z)}{G'(z)}=\frac{1+\alpha z}{1-\alpha z}+\frac{1+\beta z}{1-\beta z}-\frac{1+\alpha\beta z^2}{1-\alpha\beta z^2}\quad (z\in \mathbb{D}).
\]
Thus for $\theta\in [0,2\pi)$
\[
\Re\left\{1+\frac{e^{i \theta}G''(e^{i \theta})}{G'(e^{i \theta})}\right\}=\frac{1-\alpha^2}{1+\alpha^2-2\alpha\cos\theta}+\frac{1-\beta^2}{1+\beta^2-2\beta\cos\theta}-\frac{1-\alpha^2\beta^2}{1+\alpha^2\beta^2-2\alpha\beta\cos2\theta}.
\]
A convexity result yield from the estimating the value of the function  $g(t)$ of the variable $t:=\cos\theta$ of the form
\[
g(t):=\frac{1-\alpha^2}{1+\alpha^2-2\alpha  t}+\frac{1-\beta^2}{1+\beta^2-2\beta  t}-\frac{1-\alpha^2\beta^2}{(1+\alpha\beta)^2-4\alpha\beta t^2},
\]
where $-1\le t\le 1$.
\newline
 In order to prove the second part of lemma, define for  $\theta\in [0,2\pi)$ the function
 \[
 Q(\theta):=\left|\mathfrak{L}_{\alpha,\beta,\gamma}\myp{re^{i \theta}}\right|^2=\frac{4(1-\gamma)^2r^2}{\myp{1+\alpha^2r^2-2\alpha\, r\cos\theta}\myp{1+\beta^2
  r^2-2\beta\, r \cos\theta}}\quad (0<r<1).
 \]
We see that $\min$ or $\max$ of $Q(\theta)$ are attained at the critical points of the above function, equivalently
 \[
8(1-\gamma)^2r^3 \sin\theta \left(4\alpha\beta r \cos\theta-(\alpha+\beta)(1+\alpha\beta r^2)\right)  =0.
 \]
For $\alpha\beta= 0$ and $\alpha+\beta\neq 0$ the only ones critical points are $\theta=0, \theta=\pi$. Next, let $\alpha\beta> 0$. Then, similarly,   $Q'(\theta)=0$ for $\theta=0$ and $\theta=\pi$ since $|(\alpha+\beta)(1+\alpha\beta r^2)/4\alpha\beta|\le 1$ does not hold.
If $\alpha+\beta>0$, then for such $\theta$  we have
 \[
-\mathfrak{L}_{\alpha,\beta,\gamma}(-r) \le \left|\mathfrak{L}_{\alpha,\beta,\gamma}\left(r e^{i\theta}\right)\right|\le \mathfrak{L}_{\alpha,\beta,\gamma}(r).
 \]
And, if $\alpha+\beta<0$, then for   we obtain
 \[
\mathfrak{L}_{\alpha,\beta,\gamma}(r) \le \left|\mathfrak{L}_{\alpha,\beta,\gamma}\left(r e^{i\theta}\right)\right|\le   -\mathfrak{L}_{\alpha,\beta,\gamma}(-r).
 \]
 For the case $\alpha\beta< 0$, the  critical points are $\theta=0, \theta=\pi$, and the solutions of the equation
\begin{equation}\label{eq1}
4\alpha\beta r \cos\theta-(\alpha+\beta)(1+\alpha\beta r^2)=0.
\end{equation}
We consider three separate cases, the first is $\alpha+\beta>0$. Then for critical points $\theta=0$ and  solutions of the equation \eqref{eq1}. In this case  we have
\[
 \frac{4(1-\gamma)r\sqrt{|\alpha\beta|}}{(\beta-\alpha)(1-\alpha\beta r^2)} \le \left|\mathfrak{L}_{\alpha,\beta,\gamma}\left(r e^{i\theta}\right)\right|\le \mathfrak{L}_{\alpha,\beta,\gamma}(r).
 \]
The second case is $\alpha+\beta<0$. Then for critical points $\theta=\pi$ and  solutions of the equation \eqref{eq1},  we obtain
 \[
 \frac{4(1-\gamma)r\sqrt{|\alpha\beta|}}{(\beta-\alpha)(1-\alpha\beta r^2)} \le \left|\mathfrak{L}_{\alpha,\beta,\gamma}\left(r e^{i\theta}\right)\right|\le- \mathfrak{L}_{\alpha,\beta,\gamma}(-r).
 \]
Finally for the case $\alpha+\beta=0$, the  critical points are $\theta=0, \theta=\pi/2, \theta=\pi, \theta=3\pi/2$ and $\theta=2\pi$.  For such $\theta$  we conclude
\[
\frac{2(1-\gamma)r}{1+\alpha^2r^2}=\left|\mathfrak{L}_{\alpha,-\alpha,\gamma}(\pm i r)\right| \le \left|\mathfrak{L}_{\alpha,-\alpha,\gamma}\left(r e^{i\theta}\right)\right|\le \left|\mathfrak{L}_{\alpha,-\alpha,\gamma}(\pm r)\right|=\frac{2(1-\gamma)r}{1-\alpha^2r^2}.
\]
 \end{proof}
 Let $\alpha=\pm\beta$. Form \eqref{t_0}, the function $\mathfrak{L}_{\alpha,\alpha,\gamma}(z)/(2-2\gamma)$ is univalent in $\mathbb{D}$ if
 $t_0(\alpha,\alpha)=\frac{1+\alpha^2-4|\alpha|}{1-\alpha^2}\ge0$ and this is equivalent to the range $-2+\sqrt{3}\le \alpha\le 2-\sqrt{3}$. 
 %=====================================

 Now, we define a family of functions related to the Pascal snail $\mathcal{T}_{\alpha,\beta,\gamma}$ and present various relations of that family with the previously known classes.
\begin{defn}\label{def1}
For $-1 < \alpha\le\beta< 1$, and $0\le \gamma<1$ with $\gamma \ge \mathcal{T}_0(\alpha,\beta)$ and $\mathcal{T}_0(\alpha,\beta)$ defined by
\begin{equation}\label{eq_Con_RE}
\mathcal{T}_0(\alpha,\beta)=\left\{\begin{array}{lcl}
1-\frac{(1-\alpha\beta)[2\, \sqrt{\alpha\beta(1-\alpha^2)(1-\beta^2)}+(\alpha+\beta)(1-\alpha\beta)]}{(1+\alpha\beta)^2}&\textit{for}& (\alpha,\beta)\in B_1, \\
1-\frac{(1+\alpha)(1+\beta)}{2}& & otherwise,
\end{array}\right.
\end{equation}
 let $\mathcal{ST}_{snail}(\alpha,\beta,\gamma)$ denote the subfamily of $\mathcal{S}$ consisting of the functions $f$, satisfying the condition
\begin{equation}\label{def_main}zf'(z)/f(z)\prec \mathcal{T}_{\alpha,\beta,\gamma}(z) \quad (z\in \mathbb{D}),\end{equation}
and let $\mathcal{CV}_{snail}(\alpha,\beta,\gamma)$ be a class of analytic functions $f$
such that
\begin{equation}\label{def_main1}
1+zf''(z)/f'(z)\prec \mathcal{T}_{\alpha,\beta,\gamma}(z)\quad (z\in \mathbb{D}),
\end{equation}
 where $\mathcal{T}_{\alpha,\beta,\gamma}$ is given by \eqref{eq_def_funT}.  Geometrically, the condition \eqref{def_main} and \eqref{def_main1} means that the expression ${zf'(z)}/{f(z)}$ or $1+{zf''(z)}/{f'(z)}$ lies in a domain bounded by the generalized Pascal snail $\mathcal{T}_{\alpha,\beta,\gamma}$   (Fig.~\ref{Fig3}) given by
 \begin{equation*}
\left[(u-1)^2+v^2-a(u-1)\right]^2 = c^2(u-1)^2+d^2v^2,\end{equation*}
where $a, c$ and $d$ given by \eqref{acd}.

%====================================================
By the properties of $\mathcal{T}_{\alpha,\beta,\gamma}$, given in Theorem \ref{eq_def_funT} we have
\begin{equation}\label{eq_def_main}
\Re\set{zf'(z)/f(z)} >1+\mathfrak{L}_0\quad (z\in \mathbb{D}),
\end{equation}
for $f \in \mathcal{ST}_{snail}(\alpha,\beta,\gamma)$, and for $f \in \mathcal{CV}_{snail}(\alpha,\beta,\gamma)$
\begin{equation}\label{eq_def_main1}
\Re\set{1+zf''(z)/f'(z)}>1+\mathfrak{L}_0 \quad (z\in \mathbb{D}),
\end{equation}
where $\mathfrak{L}_0= \mathfrak{L}_0(\alpha,\beta,\gamma)$ is given in Corollary \ref{Cor1_Th-Re}.

Additionally  $\mathcal{CV}_{snail}(\alpha,\beta,\gamma)\subset \mathcal{G}$ for $\gamma$ satisfying
\begin{equation}\label{eq_Con_one_dir}
\gamma \ge\gamma_0(\alpha,\beta) = \left\{\begin{array}{lcl}
1-\frac{3(1-\alpha\beta)[2\, \sqrt{\alpha\beta(1-\alpha^2)(1-\beta^2)}+(\alpha+\beta)(1-\alpha\beta)]}{2(1+\alpha\beta)^2}&\textit{for}& (\alpha,\beta)\in B_1, \\
1-\frac{3(1+\alpha)(1+\beta)}{4}& & otherwise,
\end{array}\right.
\end{equation}
where $\mathcal{G}$ is the family of function univalent, convex in one direction, and satisfying $ \Re\{1+zf''(z)/f'(z)\} > -1/2$, see  \cite{Ozaki}.

Taking into account \eqref{t_0} the function $\mathfrak{L}_{\alpha,\alpha,\gamma}(z)/(2-2\gamma)\in \mathcal{G}$ for $|\alpha|\le 4-\sqrt{13}$,   and $\mathfrak{L}_{\alpha,-\alpha,\gamma}(z)/(2-2\gamma)\in \mathcal{G}$  for $-\sqrt{6-\sqrt{33}}\le \alpha<0$, \cite{KJ}.

Summarizing,  $\mathcal{T}_{\alpha,\beta,\gamma}$ is a analytic univalent function with positive real part in $\mathbb{D}$, $\mathcal{T}_{\alpha,\beta,\gamma}(\mathbb{D})$ is symmetric with respect to the real axis, starlike with respect  to $\mathcal{T}_{\alpha,\beta,\gamma}(0)=1$ and   convex in one direction under some conditions on $\alpha$ and $\beta$. Moreover $\mathcal{T}'_{\alpha,\beta,\gamma}(0)=2(1-\gamma)>0$ hence $\mathcal{T}_{\alpha,\beta,\gamma}(\mathbb{D})$ satisfies Ma and Minda condition \cite{MM}. We refer to \cite{KME1,KME2,KSu,MSR} for a detailed discussion about similar subclasses of related to functions mapping the unit disk onto domains contained in a right halfplane and starlike with respect to $1$.

For $\beta=\alpha$ with $-1<\alpha<1$ and $\beta=-\alpha$ with $-1<\alpha<0$, the quantities $\mathcal{T}_0(\alpha,\beta)$ and $\gamma_0(\alpha,\beta)$ are the following
\begin{equation*}
\mathcal{T}_0(\alpha,\alpha)= \left\{\begin{array}{ll}
1-4\alpha\left(\frac{1-\alpha^2}{1+\alpha^2}\right)^2\quad&\textit{for}\quad 2-\sqrt{3}\le\alpha<1,\\
\frac{1-2\alpha-\alpha^2}{2}\quad&\textit{for}\quad -1<\alpha\le 2-\sqrt{3},
\end{array}\right.
\quad \quad \mathcal{T}_0(\alpha,-\alpha)=\frac{1+\alpha^2}{2}
\end{equation*}
and
\begin{equation*}
\gamma_0(\alpha,\alpha)= \left\{\begin{array}{ll}
1-6\alpha\left(\frac{1-\alpha^2}{1+\alpha^2}\right)^2\quad&\textit{for}\quad 2-\sqrt{3}\le\alpha<1,\\
\frac{1-6\alpha-3\alpha^2}{4}\quad&\textit{for}\quad -1<\alpha\le 2-\sqrt{3},
\end{array}\right. \quad \quad \gamma_0(\alpha,-\alpha)=\frac{1+3\alpha^2}{4}.
\end{equation*}
For $\gamma=1/2$, classes $\mathcal{ST}_{snail}(\alpha,\alpha,1/2)$ and $\mathcal{CV}_{snail}(\alpha,\alpha,1/2)$ are defined under the condition
$0\le \alpha\le \alpha_0$, where $\alpha_0=0.615331\ldots$ is a root of equation
$8\alpha \myp{1-\alpha^2}^2=\myp{1+\alpha^2}^2$. We also note that  $\mathcal{ST}_{snail}(\alpha,-\alpha,1/2)$ and $\mathcal{CV}_{snail}(\alpha,-\alpha,1/2)$ of starlike and convex functions of Ma-Minda  type \cite{MM}, can not be defined, because it should satisfy $\gamma\ge \frac{1+\alpha^2}{2}$ that is $\alpha^2\le 0$ which is impossible.
\begin{figure}[!ht]
\centering
\subfloat[$\alpha=1-\sqrt{2}, \gamma=2-\sqrt{2}\      (\gamma\ge\mathcal{T}_0(\alpha,-\alpha))$]{%
\includegraphics[width=0.45\textwidth]{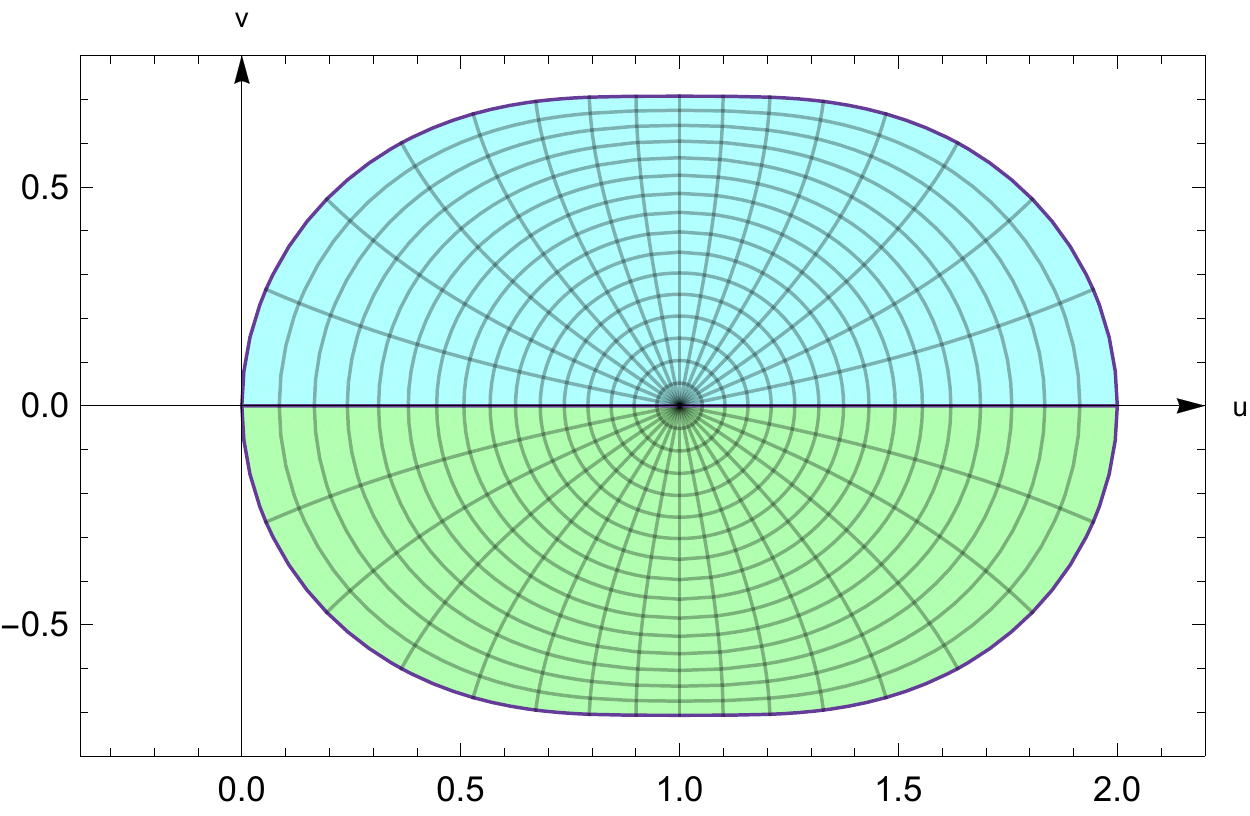}}%
\qquad\quad
\subfloat[$\alpha=1-\sqrt{2}, \gamma=1.9-\sqrt{2}\ (\gamma<\mathcal{T}_0(\alpha,-\alpha))$]{%
\includegraphics[width=0.415\textwidth]{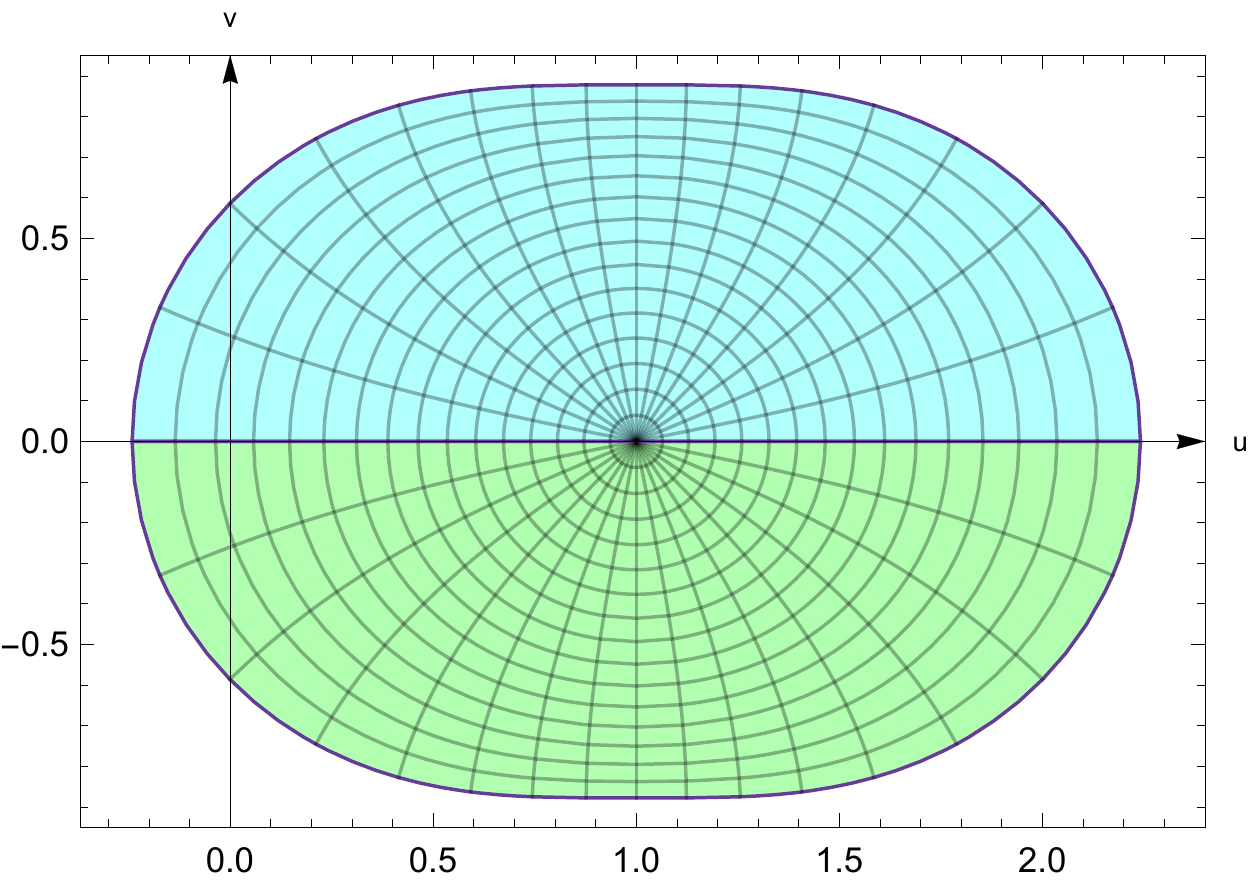}}%
\caption{Image of $\mathbb{D}$ under $\mathcal{T}_{\alpha,-\alpha,\gamma}(z)$}\label{Fig3}
\end{figure}

Further properties of $\mathcal{T}_{\alpha,\beta,\gamma}$  yield:
 \begin{align*}
 &f\in \mathcal{ST}_{snail}(\alpha,\beta, \gamma) \Longrightarrow \phi(z):=\int_{0}^{z}\left(\frac{f(t)}{t}\right)^{-\frac{1}{\mathfrak{L}_0}}dt\in \mathcal{CV}.
\end{align*}
Indeed, by logarithmic differentiation of  $\phi'(z)=\left(\frac{f(z)}{z}\right)^{-\frac{1}{\mathfrak{L}_0}}$  we obtain
\[
1+\frac{z\phi''(z)}{\phi'(z)}=1-\frac{1}{\mathfrak{L}_0}\left(\frac{zf'(z)}{f(z)}-1\right)=1+\frac{1}{\mathfrak{L}_0}-\frac{1}{\mathfrak{L}_0}
\frac{zf'(z)}{f(z)}\quad (z \in \mathbb{D}).
\]
Since $f\in \mathcal{ST}_{snail}(\alpha,\beta, \gamma) $, we conclude that
\[
\Re\left\{1+\frac{z\phi''(z)}{\phi'(z)} \right\}=1+\frac{1}{\mathfrak{L}_0}-\frac{1}{\mathfrak{L}_0}\Re\left\{ \frac{zf'(z)}{f(z)}\right\}>0\quad  (z \in \mathbb{D}).
\]
The equivalence $g\in \mathcal{ST}_{snail}(\alpha,\beta,\gamma)$ if and only if $zg'(z)/g(z)\prec \mathcal{T}_{\alpha,\beta,\gamma}(z)$ allows to determine the structural formula for functions in $\mathcal{ST}_{snail}(\alpha,\beta,\gamma)$. A function $g$ is in the class $\mathcal{ST}_{snail}(\alpha,\beta,\gamma)$
if and only if there exists an analytic function  $p\prec \mathcal{T}_{\alpha,\beta,\gamma}$, such that
\begin{equation}\label{rep_ST_snail}
	g(z)=z\exp\myp{ \int_{0}^{z} \frac{p(t)-1}{t}\, dt}.
\end{equation}
 \end{defn}
 %-------------------------------------------------------
\noindent The above integral representation provides many examples of functions of the class $\mathcal{ST}_{snail}(\alpha,\beta,\gamma)$. Let
 $p(z)=\mathcal{T}_{\alpha,\beta,\gamma}(z^n) \in \mathcal{ST}_{snail}(\alpha,\beta,\gamma)$ for $n=1,2,\ldots$. Then, for $\alpha\neq \beta,\ n\ge 1$,   the function
\begin{align}\label{eq_psi_alpha_n}
\Psi_{\alpha,\beta,\gamma,n}(z)&=z\exp\myp{ \int_{0}^{z}\frac{2(1-\gamma)t^{n-1}}{(1-\alpha t^n)(1-\beta t^{n)}}\,dt}=z\myp{\frac{1-\beta z^n}{1- \alpha z^n}}^{\frac{2(1-\gamma)}{n\myp{\alpha-\beta}}}\\
&=z+\dfrac{2(1-\gamma)}{n}z^{n+1}+\dfrac{(1-\gamma)[2(1-\gamma)+n(\alpha+\beta)]}{n^2}z^{2n+1}+\cdots,\notag
\end{align}
is extremal  for several problems in the class  $\mathcal{ST}_{snail}(\alpha,\beta,\gamma)$.
For  $n=1$ we have
\begin{equation}\label{eq_psi_alpha}
\Psi_{\alpha,\beta,\gamma}(z):= \Psi_{\alpha,\beta,\gamma,1}(z)=z\myp{\frac{1-\beta z}{1- \alpha z}}^{\frac{2(1-\gamma)}{\alpha-\beta}},
\end{equation}
and for $\alpha=\beta$
\begin{equation}\label{eq_psi_2alpha_n}
\Psi_{\alpha,\alpha,\gamma,n}(z)=z\exp\left(\frac{2(1-\gamma)z^n}{n(1- \alpha z^n)}\right)
=z+\frac{2 (1-\gamma ) z^{n+1}}{n(1-\alpha z^n)}+\cdots
\end{equation}
and
\begin{equation}\label{eq_psi_2alpha_1} \Psi_{\alpha,\alpha,\gamma}(z):=\Psi_{\alpha,\alpha,\gamma,1}(z)=z\exp\left(\frac{2(1-\gamma)z}{1- \alpha z}\right).
\end{equation}
We note that $\Psi_{\alpha,\beta,\gamma,n}(\mathbb{D})$ is sunflower's domain (Fig.~\ref{Flower}).

\begin{figure}[h]
\centering
\includegraphics[width=0.438\textwidth]{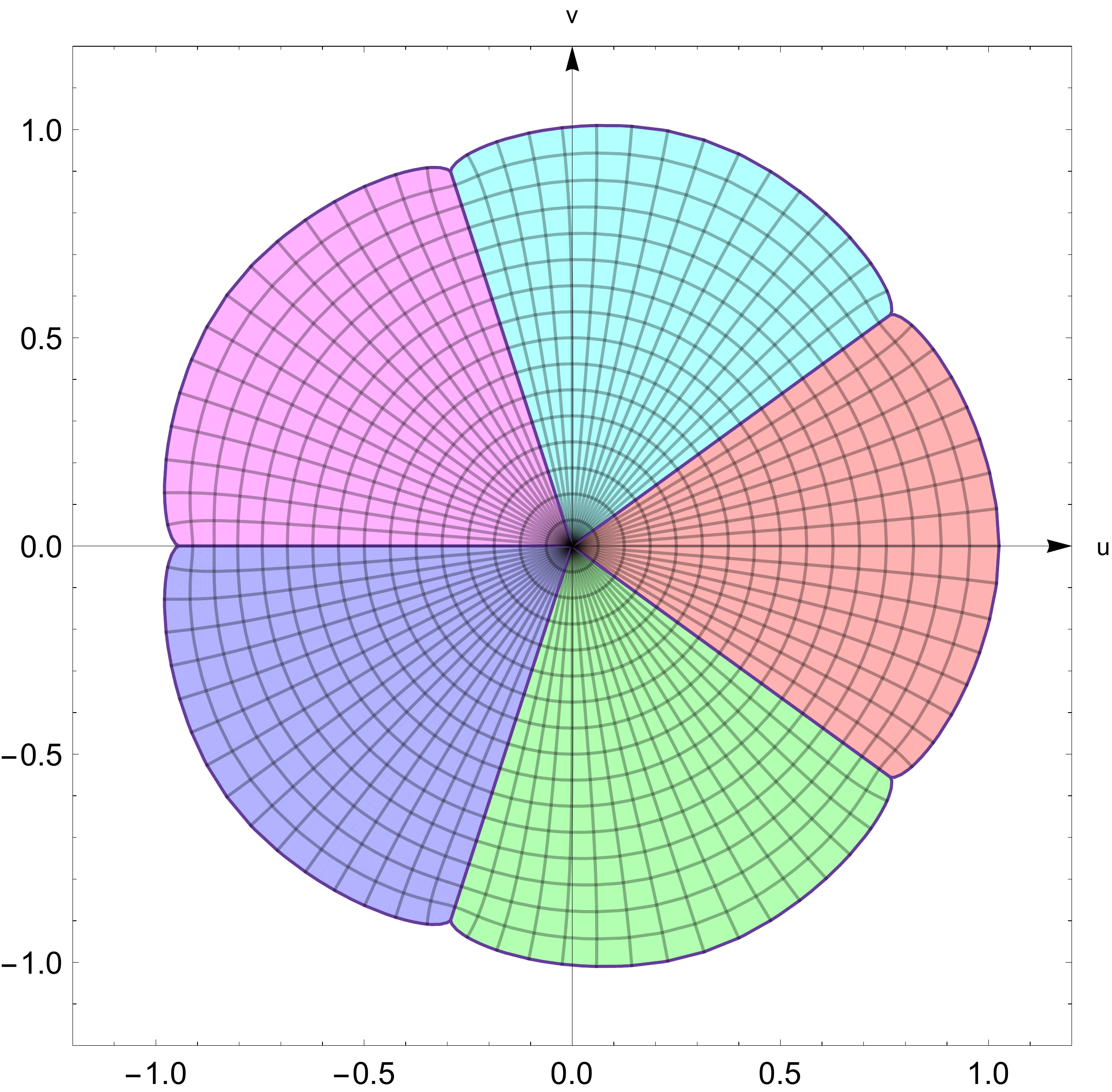}\quad\includegraphics[width=0.44\textwidth]{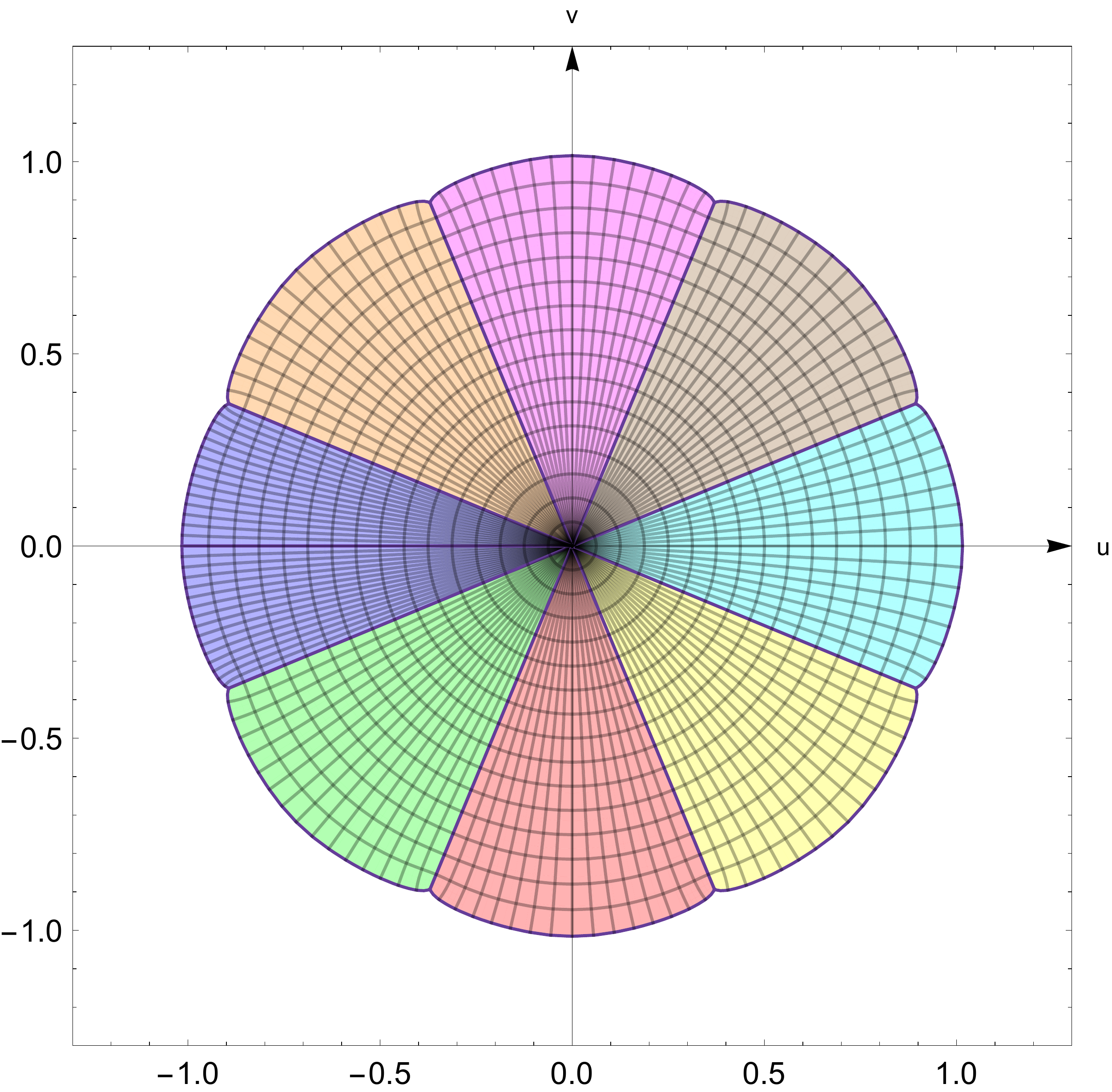}%
\caption{$\Psi_{\alpha,\beta,\gamma,n}(\mathbb{D})$ for $\alpha=-0.9,\beta=0.4,\gamma=0.93, n=5,8$.}\label{Flower}
\end{figure}
%--------------------------------------------------------

Indeed, for $\alpha \neq \beta$ let $$G(t) = \left|\Psi_{\alpha,\beta,\gamma,n}(e^{it})\right|=\left(\frac{1+\beta^2-2\beta\cos\, nt}{1+ \alpha^2-2\alpha\cos\, nt }\right)^{p},$$ where $p=\frac{1-\gamma}{n\myp{\alpha-\beta}}$ and $n\ge 2$. Since
$$G'(t) = \left(\frac{1+\beta^2-2\beta\cos\, nt}{1+ \alpha^2-2\alpha\cos\, nt }\right)^{p-1}\frac{pn(\beta-\alpha)(1-\alpha\beta)\sin\, nt}{({1+ \alpha^2-2\alpha\cos\, nt })^2},$$
we see that the points of extreme of modulus occur at $t=\frac{k\pi}{n}$, where $k=0,1,2,...,2(n-1)$. At these points $G(t)$ alternately attains its maximum and minimum, equal $\left(\frac{1+\beta}{1+\alpha}\right)^{2p}$ and  $\left(\frac{1-\beta}{1-\alpha}\right)^{2p}$, respectively.

\begin{figure}[h]
\centering
\subfloat[$\alpha=-0.9,\beta= 0.4, \gamma=0.2, n=5$ with $(\gamma<\mathcal{T}_0(\alpha,\beta))$]{%
\includegraphics[width=0.45\textwidth]{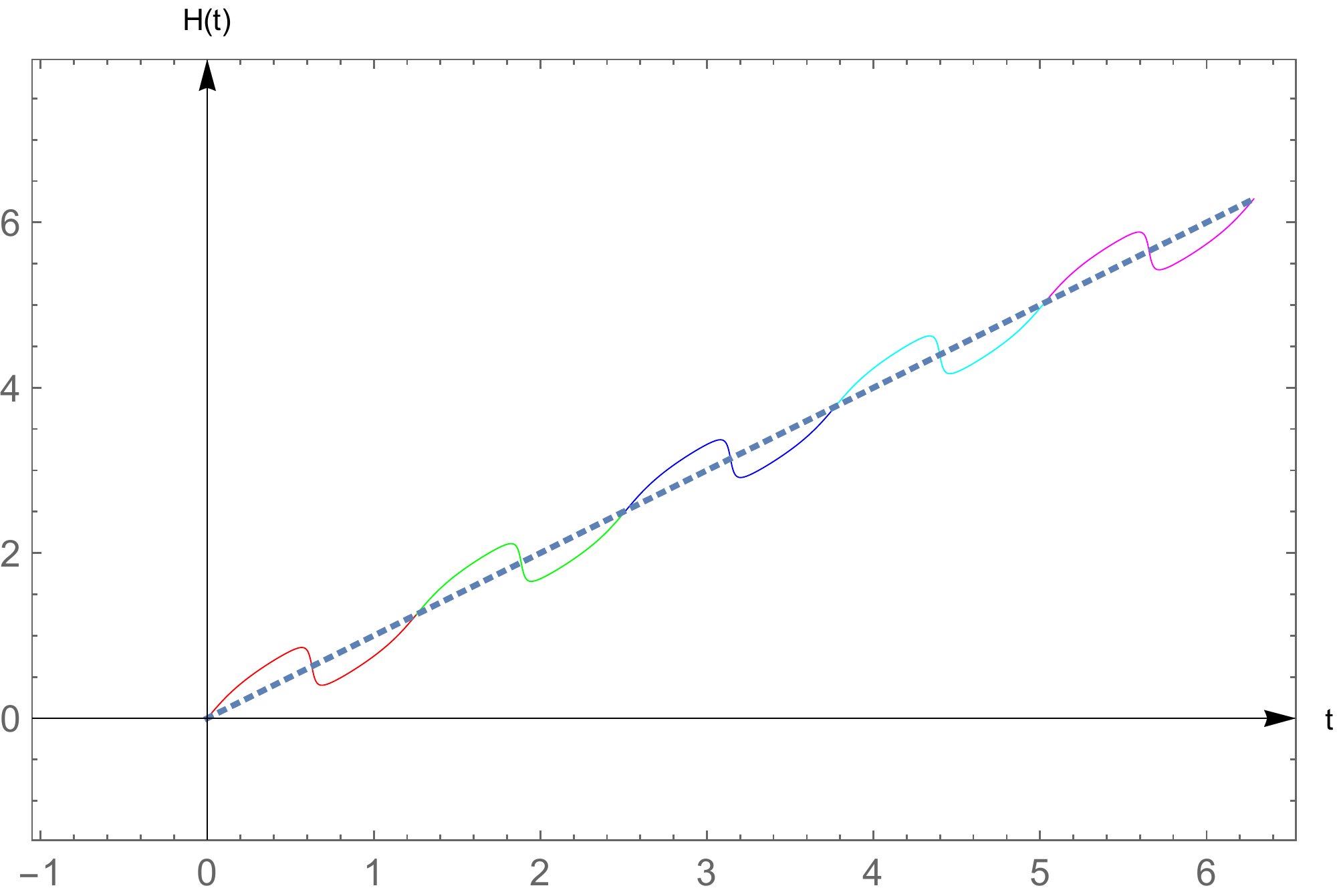}}%
\qquad\quad
\subfloat[$\alpha=0.4,\beta= 0.5, \gamma=0.3, n=5$ with $(\gamma\ge\mathcal{T}_0(\alpha,\beta))$]{%
\includegraphics[width=0.45\textwidth]{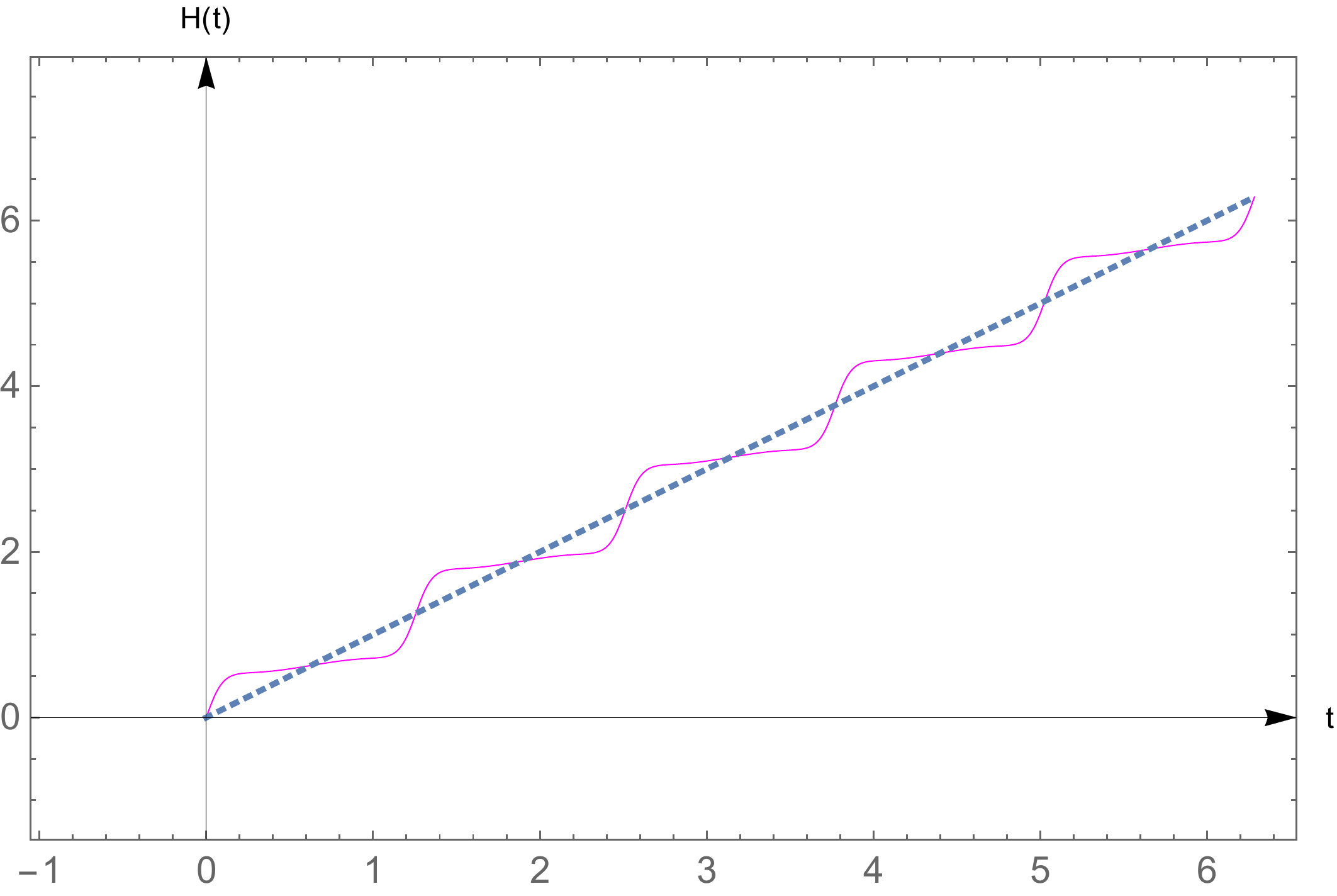}}%
\caption{Graph of function $H(t)$.}\label{Arg}
\end{figure}
Additionally, the argument of $\Psi$ i.e.
$$H(t)=\Arg\Psi_{\alpha,\beta,\gamma,n}(e^{it})=t-2p\, \tan^{-1}\frac{(\beta-\alpha)\sin\, nt}{1+\alpha\beta-(\alpha+\beta)\cos\, nt},$$
is also alternately increasing and decreasing  (see Fig. \ref{Arg}) as $t\in [0,2\pi)$  and $\gamma>\mathcal{T}_0(\alpha,\beta)$, where $\mathcal{T}_0(\alpha,\beta)$ is  defined by \eqref{eq_Con_RE}.

In the case $\alpha=\beta$ the function $G(t)$ has the form
$$G(t)=\exp\Re\left(\frac{2(1-\gamma)z^n}{n(1- \alpha z^n)}\right)= \exp\left(\frac{2(1-\gamma)(\cos nt -\alpha)}{n(1+\alpha^2-2\alpha\cos nt)}\right),$$
whose behavior is similar to the behavior of $G(t)$ for $\alpha\neq\beta$. The same situation holds for $H(t), \alpha =\beta$.

If $\gamma<\mathcal{T}_0(\alpha,\beta)$, the function $\Psi_{\alpha,\beta,\gamma,n}$ is not starlike in a whole unit disk as well as not univalent there (Fig. \ref{Fig4}).
\begin{figure}[!ht]
\centering
\includegraphics[width=0.45\textwidth]{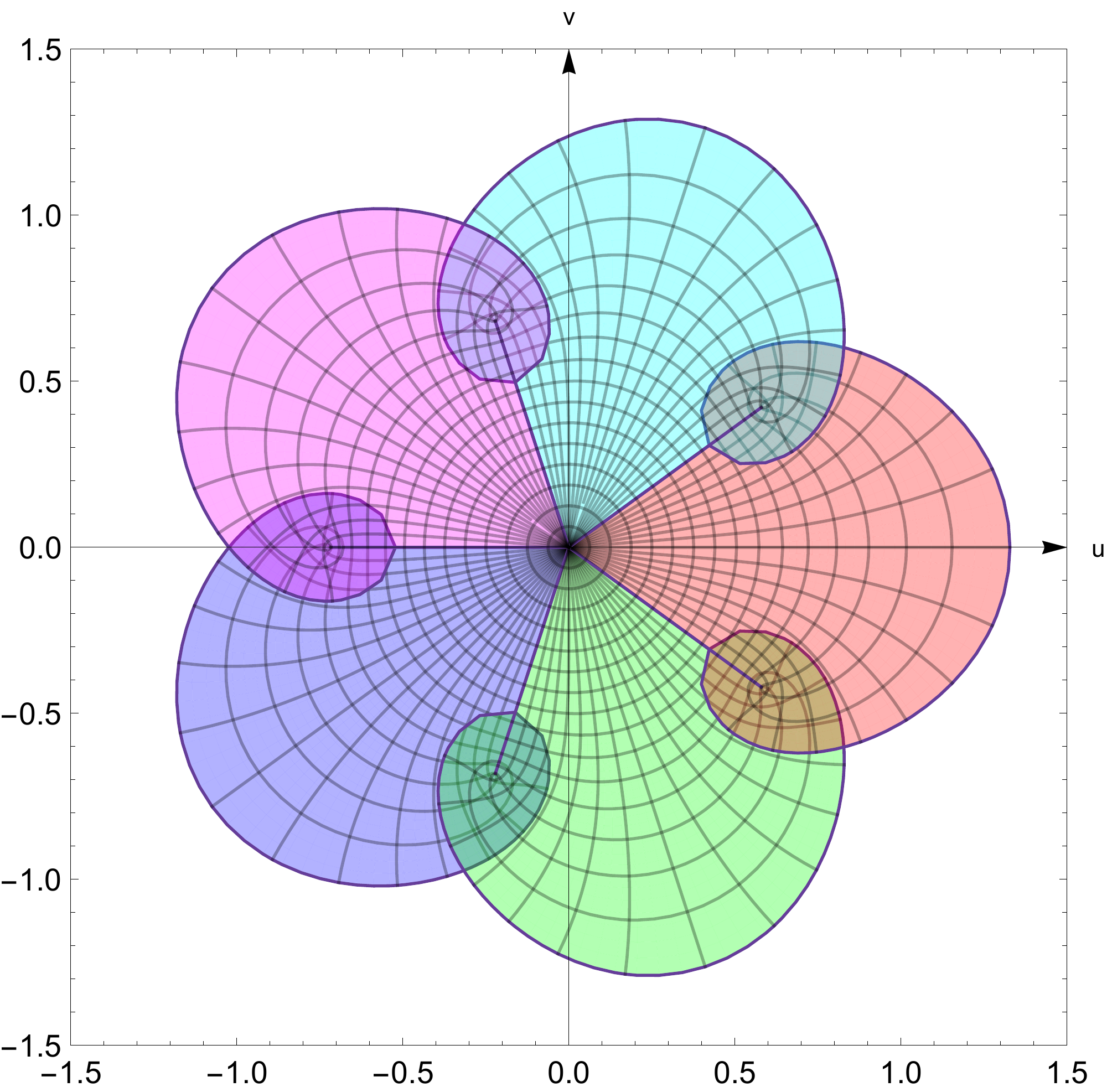}\quad\includegraphics[width=0.45\textwidth]{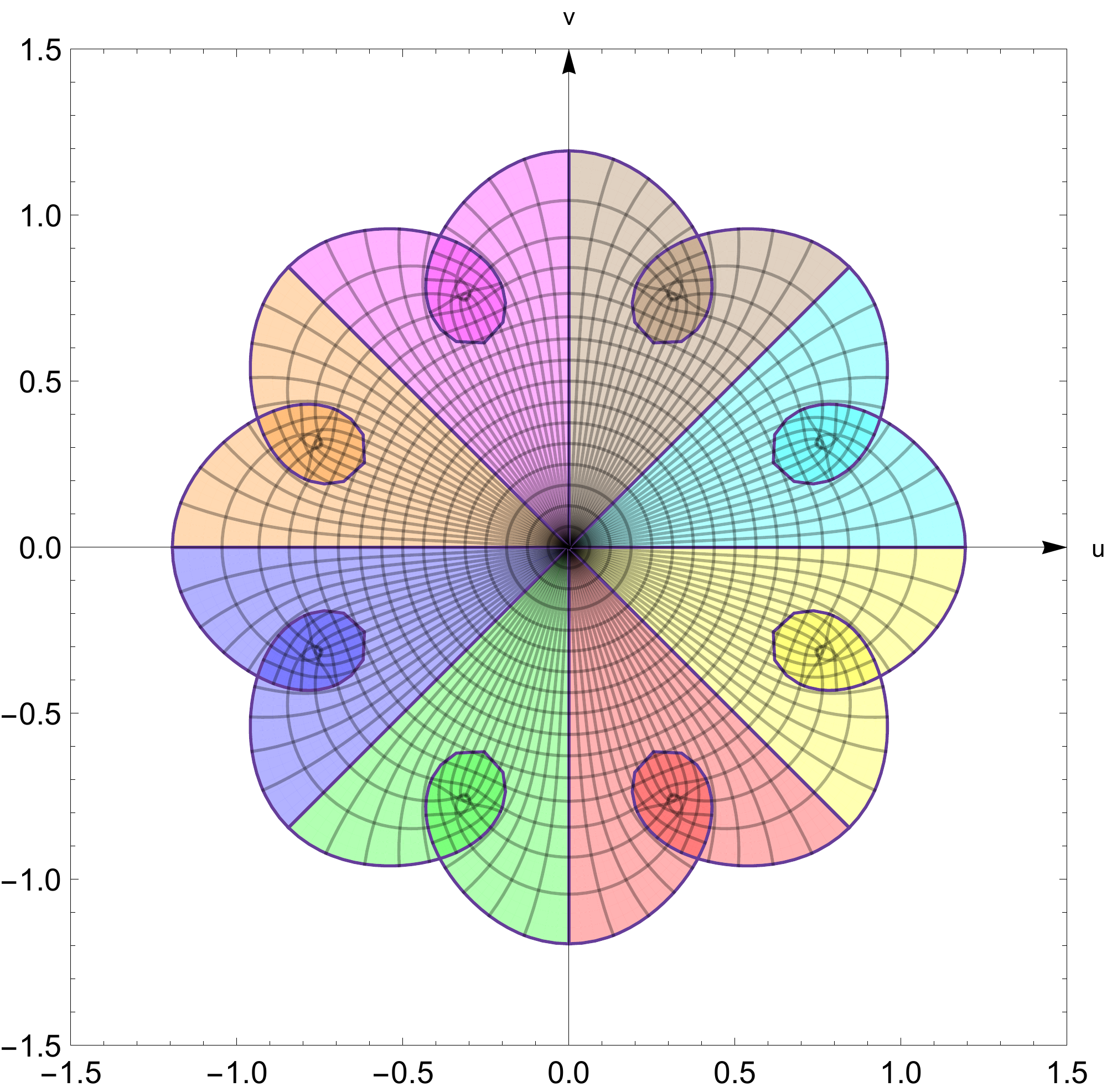}%
\caption{$\Psi_{\alpha,\beta,\gamma,n}(\mathbb{D})$ for $\alpha=-0.9,\beta=0.4,\gamma=0.2, n=5,8$ with $\gamma<\mathcal{T}_0(\alpha,\beta)$.}\label{Fig4}
\end{figure}

\noindent From Lemma \ref{univalent} it can be seen that the smallest disk with center  $(1,0)$ that contains $\mathcal{T}_{\alpha,\beta,\gamma}(z)$ and the largest disk with center at $(1,0)$ contained in $\mathcal{T}_{\alpha,\beta,\gamma}(z)$ (see Fig~\ref{Conchoid5}) are, below.
\begin{prop} Let $-1<\alpha\le\beta <1, \ \alpha\beta\neq\pm 1$. Then
\begin{equation}\label{eq_Between1}
\mathcal{T}_{\alpha,\beta,\gamma}(\mathbb{D})\supset \left\{\begin{array}{ll}
 \set{w\in \mathbb{C}\colon  |w-1|<\frac{2(1-\gamma)}{(1+\alpha)(1+\beta)}}\quad\textit{for}\quad\alpha\beta>0 \ \textit{with}\ \alpha+\beta<0\ \textit{or} \ \beta=0,\\
 \set{w\in \mathbb{C}\colon  |w-1|<\frac{2(1-\gamma)}{(1-\alpha)(1-\beta)}}\quad\textit{for}\quad \alpha\beta>0 \ \textit{with}\ \alpha+\beta>0\ \textit{or}\  \alpha=0,\\
 \set{w\in \mathbb{C}\colon  |w-1|<\frac{4(1-\gamma)\sqrt{|\alpha\beta|}}{(\beta-\alpha)(1-\alpha\beta)}}\quad\textit{for}\quad \alpha\beta<0\ \textit{with}\ \alpha+\beta\ne 0,\\
  \set{w\in \mathbb{C}\colon  |w-1|<\frac{2(1-\gamma)}{1+\alpha^2}}\quad\textit{for}\quad\alpha+\beta=0,
 \end{array}\right.
\end{equation}
\begin{equation}\label{eq_Between2}
 \mathcal{T}_{\alpha,\beta,\gamma}(\mathbb{D})\subset  \left\{\begin{array}{ll}\set{w\in \mathbb{C}\colon  |w-1|<\frac{2(1-\gamma)}{(1-\alpha)(1-\beta)}}\quad\textit{for}\quad \alpha\beta\neq0\ \textit{with}\ \alpha+\beta>0\ \textit{or}\ \alpha=0,\\
 \set{w\in \mathbb{C}\colon  |w-1|<\frac{2(1-\gamma)}{(1+\alpha)(1+\beta)}}\quad\textit{for}\quad \alpha\beta\neq0\ \textit{with}\ \alpha+\beta<0\  \textit{or}\ \beta=0,\\
\set{w\in \mathbb{C}\colon  |w-1|<\frac{2(1-\gamma)}{1-\alpha^2}}\quad\textit{for}\quad \alpha+\beta=0.
  \end{array}\right.
\end{equation}
\begin{figure}[h]
	\centering
	\includegraphics[width=0.4\textwidth]{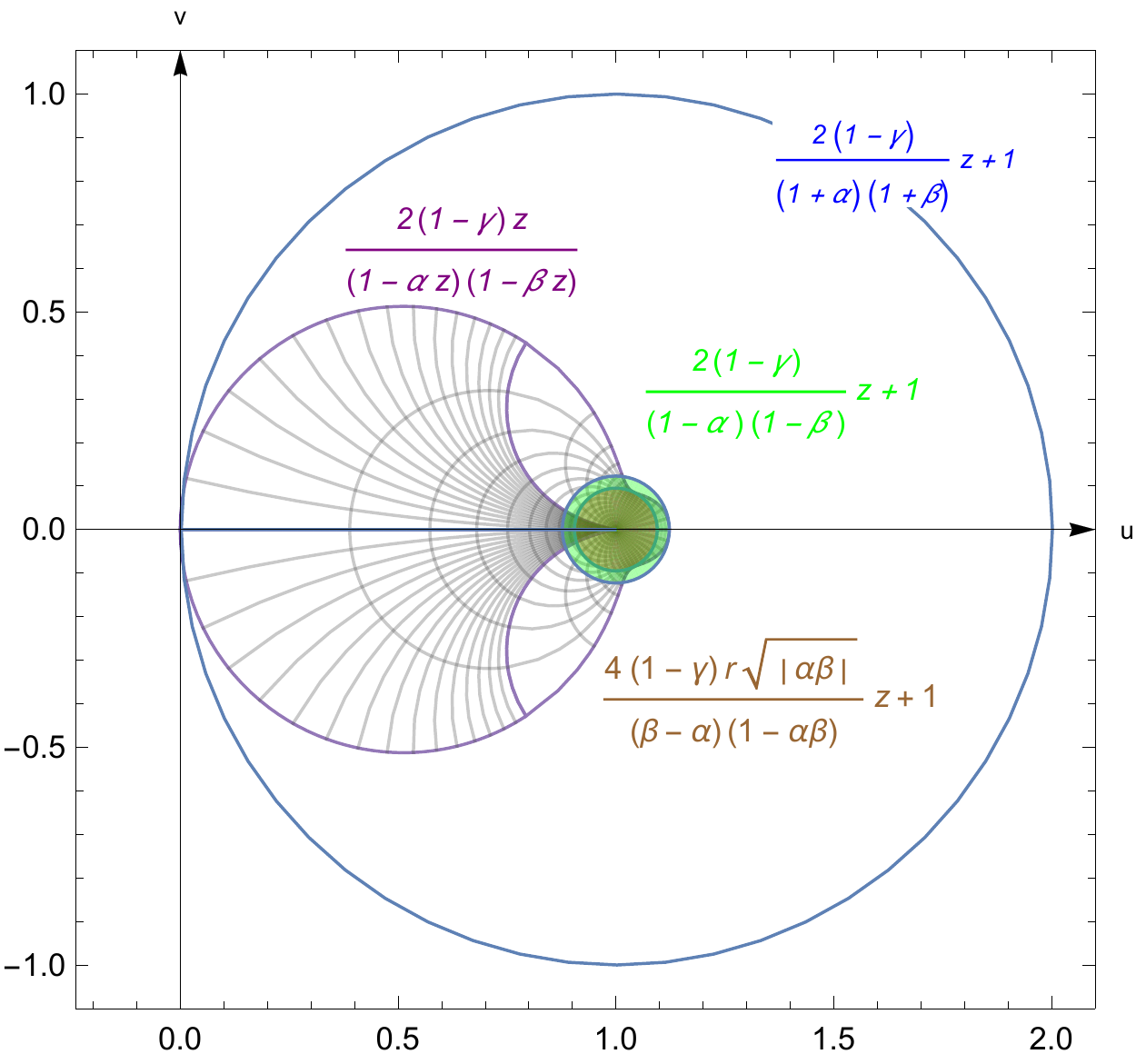}%
	\caption{The range of the functions $\mathfrak{L}_{\alpha,\beta,\gamma}$, $\frac{2(1-\gamma)}{(1+\alpha)(1+\beta)}z+1$, $\frac{2(1-\gamma)}{(1-\alpha)(1-\beta)}z+1$ and $\frac{4(1-\gamma)\sqrt{|\alpha\beta|}}{(\beta-\alpha)(1-\alpha\beta)}z+1$ for $\alpha=-0.9$, $\gamma=0.93$ and $\beta=0.4$}\label{Conchoid5}
\end{figure}
The function $\Psi_{\alpha,\beta,\gamma}$  given by \eqref{eq_psi_alpha_n}, and \eqref{eq_psi_2alpha_1} shows that the bounds are the best possible.\end{prop}
%---------------------------------
\begin{thm}\label{th_ST_con}
Let $-1 < \alpha\le\beta< 1$,  and let  $f$ be analytic in $\mathbb{D}$. If $P_f=f/f'\in \mathcal{ST}_{snail}(\alpha,\beta,\gamma)$, then
\[
\frac{zf'(z)}{f(z)}\prec\frac{z}{\Psi_{\alpha,\beta,\gamma}(z)}\quad (z \in \mathbb{D}).
\]
\end{thm}
  %---------------------------------------------------------
\begin{proof}
Let $p(z)=zf'(z)/f(z)$. Then $P_f(z)=z/p(z)$ and $zP'_f/P_f=1-zp'/p$. Since $P_f\in \mathcal{ST}_{snail}(\alpha,\beta,\gamma)$, we have
\[
-\frac{zp'(z)}{p(z)}\prec  \mathcal{T}_{\alpha,\beta,\gamma}(z)-1= \mathfrak{L}_{\alpha,\beta,\gamma}(z)\quad (z \in \mathbb{D}).
\]
The function $F$ defined by
\[
F(z)=\int_{0}^{z}\frac{\mathfrak{L}_{\alpha,\beta,\gamma}(t) }{t}\,dt=\log\myp{\frac{ \Psi_{\alpha,\beta,\gamma}(z)}{z}}
\]
 where $ \Psi_{\alpha,\beta,\gamma}$ given by \eqref{eq_psi_alpha}, is analytic in $\mathbb{D}$, $F(0)=F'(0)-1=0$ and
 \[
 1+\frac{zF''(z)}{F'(z)}=\frac{z \mathfrak{L}'_{\alpha,\beta,\gamma}(t) }{\mathfrak{L}_{\alpha,\beta,\gamma}(t)} \quad (z \in \mathbb{D}).
\]
Taking into account Lemma \ref{univalent}, we deduce that the function $F$ is convex in $\mathbb{D}$. Applying \cite{Suff}, we conclude that
\[
-\log p(z)\prec \log\myp{\frac{\Psi_{\alpha,\beta,\gamma}(z)}{z}}\quad \textit{or}\quad \log p(z)\prec \log\myp{\frac{z}{\Psi_{\alpha,\beta,\gamma}(z)}},
\]
and by \eqref{eq_psi_alpha}, the required result follows.
\end{proof}
Since for  $z\in \mathbb{D}$ and $\alpha\neq \beta$
\[
\Re\left\{\frac{z}{\Psi_{\alpha,\beta,\gamma}(z)}\right\}=\left|\frac{1-\alpha z}{1-\beta z}\right|^{-\frac{2(1-\gamma)}{\beta-\alpha}}\cos\left(\frac{2(1-\gamma)}{\beta-\alpha}\,\arg \frac{1-\alpha z}{1-\beta z} \right)
\]
and $\arg \frac{1-\alpha z}{1-\beta z}\in (-\pi/2,\pi/2)$ the above and Theorem \ref{th_ST_con} leads to the following conclusion.
  \begin{cor}\label{cor_ST_1}
	Let $f\in \mathcal{A}$ be a locally univalent function. If $P_f=f/f'\in  \mathcal{ST}_{snail}(\alpha,\beta,\gamma)$ with $\alpha\neq\beta$ and $\gamma\ge1-\frac{\beta-\alpha}{2}$, then $f\in \mathcal{ST}$.
\end{cor}
%From Corollary \ref{cor_ST_1}, we obtain the starlikeness condition for analytic functions of the unit disk. If
%\[
%\Re\left\{\frac{zf'(z)}{f(z)}-\frac{zf''(z)}{f'(z)}\right\}>1+\mathfrak{L}_0\quad\left(z\in \mathbb{D}\right)
%\]
%and $\gamma\ge1-\frac{\beta-\alpha}{2}$, then $f\in \mathcal{ST}$.
%===============================================

Now we get a representation of functions in class $ \mathcal{ST}_{snail}(\alpha,\beta,\gamma)$ with the help of the class $\mathcal{ST}[\beta]$.
\begin{lem}\label{lem_rep_ST_snail}
Let  $f\in \mathcal{ST}_{snail}(\alpha,\beta,\gamma)$ with $\alpha, \beta\ne 0$ and $\alpha\neq\beta$. Then there exists  $h\in \mathcal{ST}[\beta]$, and $g \in \mathcal{ST}[\alpha]$ such that
\[
 f(z)=z \myp{\frac{h(z)}{g(z)}}^{\frac{2(1-\gamma)}{\beta-\alpha}} \quad (z \in \mathbb{D}).
\]
  \end{lem}
  \begin{proof}
 Let $f\in \mathcal{ST}_{snail}(\alpha,\beta,\gamma)$. Then, by \eqref{rep_ST_snail},  there exists a self-map  $\omega$, which is analytic in  $\mathbb{D}$, $\omega(0)=0,\ |\omega(z)|<1$, and  such that
\begin{align*}
f(z)&=z\exp\myp{ \int_{0}^{z} \frac{\mathcal{T}_{\alpha,\beta,\gamma}(\omega(t))-1}{t}\, dt}=z\exp\ \int_{0}^{z}q\left[ \frac{\beta \omega(t)}{t(1-\beta \omega(t))}-\frac{\alpha\omega(t)}{t(1-\alpha \omega(t))}\right]\, dt\\
&=z \left(\frac{z\exp\displaystyle\int_{0}^{z} \frac{\beta\omega(t)}{t[1-\beta \omega(t)]}\, dt}{z\exp\displaystyle\int_{0}^{z} \frac{\alpha\omega(t)}{t[1-\alpha \omega(t)]}\, dt}\right)^{q}=z \left(\frac{z\exp\displaystyle\int_{0}^{z} \frac{\frac{1}{1-\beta \omega(t)}-1}{t}\, dt}{z\exp\displaystyle\int_{0}^{z} \frac{\frac{1}{1-\alpha \omega(t)}-1}{t}\, dt}\right)^{q}=z \myp{\frac{h(z)}{g(z)}}^{q},
\end{align*}
where $q=\frac{2(1-\gamma)}{\beta-\alpha}$. The assertion now follows.
  \end{proof}
%=================================================
From the relation $h\in \mathcal{CV}_{snail}(\alpha,\beta,\gamma)$ if and only if $1+zh''(z)/h'(z)\prec
\mathcal{T}_{\alpha,\beta,\gamma}(z)$ we obtain the structural formula for functions in $\mathcal{CV}_{snail}(\alpha,\beta,\gamma)$. A function $h$ is in the class
$\mathcal{CV}_{snail}(\alpha,\beta,\gamma)$ if and only if there exists an analytic function $p$ with $p\prec \mathcal{T}_{\alpha,\beta,\gamma}$, such that
\begin{equation}\label{rep_CV_snail}
	h(z)=\int_{0}^{z}\exp\myp{ \int_{0}^{w} \frac{p(t)-1}{t}\, dt}dw.
\end{equation}
 %-------------------------------------------------------
The above representation supply many examples of functions in class $\mathcal{CV}_{snail}(\alpha,\beta,\gamma)$. Let
$p(z)=\mathcal{T}_{\alpha,\beta,\gamma}(z) \in \mathcal{CV}_{snail}(\alpha,\beta,\gamma)$, then for some $n\ge 1$ and $\alpha\neq\beta$,  the
 functions
\begin{multline}\label{eq_K_alpha_n}
K_{\alpha,\beta,\gamma,n}(z)=\int_{0}^{z}\exp\myp{ \int_{0}^{w}
\frac{2(1-\gamma)t^{n-1}}{\myp{1-\alpha t^n}\myp{1-\beta t^{n}}}\,
   dt}dw =\int_{0}^{z}\myp{\frac{1-\alpha t^n}{1- \beta t^n}}^{\frac{2(1-\gamma)}{n\myp{\alpha-\beta}}}dt,
\end{multline}
are extremal functions for several problems in the class $\mathcal{CV}_{snail}(\alpha,\beta,\gamma)$.
 For  $n=1$ we have
\begin{equation}\label{eq_K_alpha}
K_{\alpha,\beta,\gamma}(z):=K_{\alpha,\beta,\gamma,1}(z)=\int_{0}^{z}\myp{\frac{1-\alpha t}{1-\beta t}}^{\frac{2(1-\gamma)}{\alpha-\beta}}dt.
\end{equation}
and for $\alpha=\beta$
\begin{equation}\label{eq_K_alpha_1}
K_{\alpha,\alpha,\gamma,n}(z)=\int_{0}^{z}\exp\left(\frac{2(1-\gamma)t^n}{n(1- \alpha t^n)}\right)dt\quad\textit{and}\quad K_{\alpha,\alpha,\gamma}(z):=K_{\alpha,\alpha,\gamma,1}(z).
\end{equation}
%=======================================

Now we get a representation of functions in class $\mathcal{CV}_{snail}(\alpha,\beta,\gamma)$ with the help of class $\mathcal{ST}[\beta]$.
From Lemma \ref{lem_rep_ST_snail}, we conclude the following Corollary.
\begin{cor}\label{cor_rep_CV_snail}
Let  $f\in \mathcal{CV}_{snail}(\alpha,\beta,\gamma)$ with $\alpha, \beta\ne 0$ and $\alpha\neq \beta$. Then there exists  $h\in \mathcal{ST}[\beta]$ and $g\in \mathcal{ST}[\alpha]$ such that
\[
 f'(z)=\myp{\frac{h(z)}{g(z)}}^{\frac{2(1-\gamma)}{\beta-\alpha}}\quad (z \in \mathbb{D}).
\]
  \end{cor}
  %===================================
   From \eqref{eq_Between1}, we conclude that  $f\in \mathcal{ST}_{snail}(\alpha,\beta,\gamma)$  if and only if
    \[
  \left|\dfrac{zf'(z)}{f(z)}-1\right|<L =
  \left\{   \begin{array}{ll}
   \frac{2(1-\gamma)}{(1+\alpha)(1+\beta)} \quad&\textit{for}\quad\alpha\beta>0 \ \textit{with}\ \alpha+\beta<0\ \textit{or} \ \beta=0,\\
 \frac{2(1-\gamma)}{(1-\alpha)(1-\beta)} \quad&\textit{for}\quad\alpha\beta>0 \ \textit{with}\ \alpha+\beta>0\ \textit{or} \ \alpha=0,\\
  \frac{4(1-\gamma)\sqrt{|\alpha\beta|}}{(\beta-\alpha)(1-\alpha\beta)}\quad&\textit{for}\quad  \alpha\beta<0 \ \textit{with}\   \alpha+\beta\ne 0,\\
  \frac{2(1-\gamma)}{1+\alpha^2}\quad&\textit{for}\quad   \alpha+\beta=0
  \end{array}\right.
  \]
  and the fact that $f\in \mathcal{CV}_{snail}(\alpha,\beta,\gamma)$ if and only if $zf'(z)\in \mathcal{ST}_{snail}(\alpha,\beta,\gamma)$,  we get the following conclusions.
\begin{prop}\label{Ex} Let $-1 < \alpha\le\beta< 1$. The classes $\mathcal{ST}_{snail}(\alpha,\beta,\gamma)$ and $\mathcal{CV}_{snail}(\alpha,\beta,\gamma)$ are nonempty. The following functions are the examples of their members.\rm
\begin{enumerate}
\item
Let $a_n\in \mathbb{C}$ with $n=2,3 ,\ldots$. Then $f(z)=z+a_nz^n \in \mathcal{ST}_{snail}(\alpha,\beta,\gamma)$ $\Longleftrightarrow
\left|a_n\right|\le \frac{L}{n-1+L}.$
\item Let $a_n\in \mathbb{C}$ with $n=2, 3, \ldots$. Then  $f(z)=z+a_nz^n\in\mathcal{CV}_{snail}(\alpha,\beta,\gamma)$ $\Longleftrightarrow
n\left|a_n\right|\leq \frac{L}{n-1+L}.$
\item
Let $A\in \mathbb{C}$. Then $z/\myp{1-Az}^2\in \mathcal{ST}_{snail}(\alpha,\beta,\gamma)
\Longleftrightarrow  |A|\le \frac{2}{2+L}.$
 \item
Let $A\in \mathbb{C}$. Then $z/(1-Az) \in\mathcal{CV}_{snail}(\alpha,\beta,\gamma) \Longleftrightarrow
|A|\leq \frac{2}{2+L}.$
 \item
Let $A\in \mathbb{C}$. Then $z\exp(Az)\in \mathcal{ST}_{snail}(\alpha,\beta,\gamma) \Longleftrightarrow  |A|\le L.$
\item
Let $A\in \mathbb{C}$. Then $\frac{\exp(Az)-1}{A}\in \mathcal{CV}_{snail}(\alpha,\beta,\gamma)\Longleftrightarrow  0<|A|\le L,$\\
where $L$ is given in the Corollary \ref{cor_rep_CV_snail}.
\end{enumerate}\end{prop}
%======================================

The following corollary is the consequence of  Lemma \ref{univalent}, and Theorems in \cite{MM}.
\begin{cor}\label{corollary1} For $-1 < \alpha\le\beta< 1$, $|z|=r<1$, and $f\in \mathcal{ST}_{snail}(\alpha,\beta,\gamma)$, it holds
$$
-\Psi_{\alpha,\beta,\gamma}(-r) \leq \left|f(z) \right|\leq \Psi_{\alpha,\beta,\gamma}(r),$$
$$\Psi_{\alpha,\beta,\gamma}'(-r)\leq \left|f'(z) \right|\leq   \Psi'_{\alpha,\beta,\gamma}(r)\quad\textit{for}\quad\alpha\beta>0 \ \textit{with}\ \alpha+\beta>0\ \textit{or} \ \alpha=0,$$
$$\Psi_{\alpha,\beta,\gamma}'(r)\leq \left|f'(z) \right|\leq   \Psi'_{\alpha,\beta,\gamma}(-r)\quad\textit{for}\quad\alpha\beta>0 \ \textit{with}\ \alpha+\beta<0\ \textit{or} \ \beta=0,$$
$$\left| {\Arg}\set{{f(z)}/{z}}\right|\leq \max_{|z|=r}{\Arg}\set{{\Psi_{\alpha,\beta,\gamma}(z)}/{z}}.$$
Equalities in the above inequalities hold at a given point other than origin for the functions
\begin{equation}\label{eq_psi_alpha_lambda}
\psi_{\alpha,\gamma,\mu}(z)=\overline{\mu}\Psi_{\alpha,\beta,\gamma}(\mu z)\quad  \myp{\abs{\mu}=1}.\end{equation}
Moreover
\begin{equation}\frac{f(z)}{z}\prec \frac{\Psi_{\alpha,\beta,\gamma}(z)}{z} \quad (z \in \mathbb{D}).
\end{equation}
If $f\in \mathcal{ST}_{snail}(\alpha,\beta,\gamma)$,  then either $f$ is a rotation of $\Psi_{\alpha,\beta,\gamma}$ given by \eqref{eq_psi_alpha} and \eqref{eq_psi_2alpha_1}  or
$$\left\{w\in \mathbb{C} \colon |w|\leq  -\Psi_{\alpha,\beta,\gamma}(-1) \right\}\subset  f(\mathbb{D}),$$  where $-\Psi_{\alpha,\beta,\gamma}(-1)=\lim_{r\to 1^-}[-\Psi_{\alpha,\beta,\gamma}(-r)]$.
\end{cor}
%================================================
 \begin{cor}\label{corollary11} Let $-1 < \alpha\le\beta< 1$. If $f\in \mathcal{CV}_{snail}(\alpha,\beta,\gamma)$  and $|z|=r<1$, then
$$
-K_{\alpha,\beta,\gamma}(-r) \leq |f(z)|\leq K_{\alpha,\beta,\gamma}(r),$$
$$K_{\alpha,\beta,\gamma}'(-r) \leq |f'(z)|\leq   K'_{\alpha,\beta,\gamma}(r),$$
$$|{\Arg}\set{f'(z)}| \leq  \max_{|z|=r}{\Arg}\set{{K'_{\alpha,\beta,\gamma}(z)}}.$$
Equalities in the above inequalities hold  at a given point other than $0$ for functions
$\overline{\mu}K_{\alpha,\beta,\gamma}(\mu z)$ with $\myp{\abs{\mu}=1}.$
Moreover
$${f'(z)}\prec {K'_{\alpha,\beta,\gamma}(z)} \quad  (z \in \mathbb{D}).$$
If $f\in \mathcal{CV}_{snail}(\alpha,\beta,\gamma)$,  then either $f$ is a rotation of $K_{\alpha,\beta,\gamma}$ given by \eqref{eq_K_alpha} and \eqref{eq_K_alpha_1} or
$$\{w\in \mathbb{C} \colon |w|\leq  -K_{\alpha,\beta,\gamma}(-1)\}\subset f(\mathbb{D}),$$ where $-K_{\alpha,\beta,\gamma}(-1)=\lim_{r\to 1^-}[-K_{\alpha,\beta,\gamma}(-r)]$.
\end{cor}
%----------------------------------------------------
\begin{thm}\label{th:twoinequality}Let $-1 < \alpha<\beta< 1$. If $f\in \mathcal{ST}_{snail}(\alpha,\beta,\gamma)$, then
\begin{enumerate}\rm
     \item\label{P_1}
          $
          \Re\set{\dfrac{f(z)}{z}}>\myp{\dfrac{1+\alpha}{1+\beta}}^{\frac{2-2\gamma}{\beta-\alpha}}\quad\textit{for}\quad
           \mathcal{T}_0(\alpha,\beta)\le 1-\dfrac{\beta-\alpha}{2} \le \gamma\ (z\in
          \mathbb{D}),$
           \item\label{P_11}
          $
          \Re\set{\dfrac{f(z)}{z}}^{\frac{\beta-\alpha}{2-2\gamma}}>\dfrac{1+\alpha}{1+\beta}\quad  (z \in \mathbb{D}),$
          \item\label{P_2}
     $ \left|{\Arg}\set{\dfrac{f(z)}{z}}\right|\le\dfrac{2(1-\gamma)}{\beta-\alpha}
     \sin^{-1}\myp{\dfrac{|z|(\beta-\alpha)}{1-|z|^2\alpha\beta}}\quad  (z \in \mathbb{D}).$
     \end{enumerate}
\end{thm}
 \begin{proof}
 \noindent Let  $q:=\frac{2(1-\gamma)}{\beta-\alpha}$.\\
 \underline{\rm Case \ref{P_1}}.
 From $1-\frac{\beta-\alpha}{2} \le \gamma$ it follows that  $0<(2-2\gamma)/(\beta-\alpha)\le1$, and from $f\in \mathcal{ST}_{snail}(\alpha,\beta,\gamma)$ it follows that $\mathcal{T}_0\le 1-\frac{\beta-\alpha}{2}$. Then, making use  Corollary \ref{corollary1} and Lemma \ref{inequalityreal}, we conclude that
\begin{align*}
\Re\set{\frac{f(z)}{z}}>\Re\set{\frac{\Psi_{\alpha,\beta,\gamma}(z)}{z}}=\Re\set{\myp{\frac{1-\alpha z}{1-\beta z}}^{q}}
\ge\set{\Re\myp{\frac{1-
     \alpha z}{1- \beta z}}}^{q}>\myp{\frac{1+\alpha}{1+\beta}}^{q}.
\end{align*}
The function $\psi_{\alpha,\gamma,\mu}$ given by
\eqref{eq_psi_alpha_lambda}, shows that the bound is the best possible.\\
\underline{\rm Case \ref{P_11}}. From Corollary \ref{corollary1} we have
\[
\left[\frac{f(z)}{z}\right]^{1/q}\prec \left[\frac{\Psi_{\alpha,\beta,\gamma}(z)}{z}\right]^{1/q}.
\]
Thus
\[
\Re\set{\frac{f(z)}{z}}^{1/q}>\Re\set{\frac{1-\alpha z}{1-\beta z}}>\frac{1+\alpha}{1+\beta}.
\]
\underline{\rm Case  \ref{P_2}}.  By Corollary \ref{corollary1}  it is enough to consider ${\Arg}\set{{\Psi_{\alpha,\beta,\gamma}(z)}/{z}}$.  Since the image of the disk
 $\set{z\in \mathbb{C}\colon |z|\le r}$ by the function
 $w={\Psi_{\alpha,\beta,\gamma}(z)}/{z}$ or
 $w^{1/q}=\myp{1-\alpha z}/\myp{1-\beta z}$   is contained in
 closed disc with center $\myp{1-\alpha\beta r^2}/\myp{1-\beta^2 r^2}$  and  radius
 $\myp{r(\beta-\alpha)}/\myp{1-\beta^2 r^2}$. Therefore
\begin{equation*}
\left|w^{1/q}-\frac{1-\alpha\beta r^2}{1-\beta^2 r^2}\right|\le
\frac{r(\beta-\alpha)}{1-\beta^2 r^2}\quad \textit{and}\quad \left|{\Arg}\,w^{1/q}\right|<\frac{\pi}{2}.
\end{equation*}
Thus
\[
\left|{\Arg}\,w^{1/q}\right|\le \sin^{-1}\left(\frac{r(\beta-\alpha)}{1-r^2\alpha\beta}\right).
\]
The proof is now complete.
\end{proof}
%==========================================
It is clear  that  $f(z)\in \mathcal{CV}_{snail}(\alpha,\beta,\gamma)$ if and only if  $zf'(z)\in
\mathcal{ST}_{snail}(\alpha,\beta,\gamma)$. Using the same notation and the same
reasoning as in the
proof of Theorem \ref{th:twoinequality} we have the  following Corollary.
\begin{cor}\label{th:twoinequality2}Let $-1 < \alpha<\beta< 1$. If $f\in \mathcal{CV}_{snail}(\alpha,\beta,\gamma)$, then
\begin{enumerate}\rm
 \item
  $
          \Re\set{f'(z)}>\myp{\dfrac{1+\alpha}{1+\beta}}^{\frac{2-2\gamma}{\beta-\alpha}}\quad\textit{for}\quad
          \mathcal{T}_0(\alpha,\beta)\le 1-\dfrac{\beta-\alpha}{2} \le \gamma\ (z\in
          \mathbb{D}),$
           \item
          $
          \Re\set{f'(z)}^{\frac{\beta-\alpha}{2-2\gamma}}>\dfrac{1+\alpha}{1+\beta}\quad (z \in \mathbb{D}),$
          \item
     $ \abs{{\Arg}\set{f'(z)}}\le\dfrac{2(1-\gamma)}{\beta-\alpha}
     \sin^{-1}\left(\dfrac{|z|(\beta-\alpha)}{1-|z|^2\alpha\beta}\right)\quad (z \in \mathbb{D}).$
     \end{enumerate}
\end{cor}
Condition $ \mathcal{T}_0(\alpha,\beta)\le 1-\dfrac{\beta-\alpha}{2}\le \gamma$ in Theorem \ref{th:twoinequality} for  requirement $\beta=-\alpha$ are equivalent to conditions $1-\sqrt{2}\le\alpha<0$, $\gamma\ge 1+\alpha$, and so we have the following.
\begin{cor} For $1-\sqrt{2}\le\alpha<0$ and  $\gamma\ge 1+\alpha$, we have:
\[
f\in \mathcal{ST}_{snail}(\alpha,-\alpha,\gamma) \Longrightarrow		\Re\set{\dfrac{f(z)}{z}}>\myp{\dfrac{1-\alpha}{1+\alpha}}^{\frac{1-\gamma}{\alpha}}\quad (z\in \mathbb{D}),\]
			and
			\[
	f\in \mathcal{CV}_{snail}(\alpha,-\alpha,\gamma) \Longrightarrow			\Re\set{f'(z)}>\myp{\dfrac{1-\alpha}{1+\alpha}}^{\frac{1-\gamma}{\alpha}}\quad	(z\in \mathbb{D}).\]
\end{cor}
\noindent\textbf{Acknowledgments}

\noindent  The authors thank the editor and the anonymous referees for constructive and pertinent suggestions.\bigskip

\noindent\textbf{Availability of supporting data}

\noindent Not applicable.\bigskip

\noindent\textbf{Competing Interests}

\noindent The authors declare that they have no competing interests.\bigskip

\noindent\textbf{Funding}

\noindent This work was partially supported by the Center for Innovation and Transfer of Natural Sciences and Engineering
Knowledge, Faculty of Mathematics and Natural Sciences, University of Rzeszow.\bigskip

\noindent\textbf{Authors' Contributions}

\noindent Each of the authors contributed to each part of this study equally, all authors read and approved the final manuscript.

%===============================================

\end{document}